\documentclass[11pt]{article}
\usepackage{amsfonts}
\usepackage{graphics,amssymb,amsthm,amsmath,epsf}
\usepackage{graphicx}
\usepackage{subfigure}

\usepackage[round]{natbib}
\usepackage{xcolor}

\newtheorem{theorem}{Theorem}[section]
\newtheorem{prop}{Proposition}

\newtheorem{corollary}{Corollary}[section]
\numberwithin{equation}{section}
\newtheorem{remark}{Remark}[section]

\allowdisplaybreaks

\begin{document}

\title{An Extension of the Generalized Linear Failure Rate Distribution}
\author{  M. R. Kazemi\thanks{Corresponding: kazemi@fasau.ac.ir} , A. A. Jafari$^{\dag}$, S. Tahmasebi$^{\ddag}$  \\
{\small $^*$Department of Statistics, Fasa University, Fasa, Iran}\\
{\small $^\dag$Department of Statistics, Yazd University, Yazd, Iran}\\
{\small $^\ddag$Department of Statistics, Persian Gulf University, Bushehr, Iran}\\
}
\date{}
\maketitle
\begin{abstract}
In this paper, we introduce a new extension of the generalized linear failure rate distributions. It includes some well-known lifetime distributions such as extension of generalized exponential and generalized linear failure rate distributions as special sub-models. In addition, it can have a constant, decreasing, increasing, upside-down bathtub (unimodal), and bathtub-shaped hazard rate function depending on its parameters. We provide some of its statistical properties such as moments, quantiles, skewness, kurtosis, hazard rate function, and reversible hazard rate function. The maximum likelihood estimation of the parameters is also discussed. At the end, a real data set is given to illustrate the usefulness of this new distribution in analyzing lifetime data.

\end{abstract}

\section{Introduction}

The generalized linear failure rate (GLFR) distribution is defined by
\cite{sa-ku-09}
and contains various well-known distributions: the generalized exponential (GE) distribution introduced by
\cite{gu-ku-99},
the generalized Rayleigh distribution introduced by
\cite{su-pa-01,su-pa-05},
the exponential, Rayleigh, and linear failure rate (LFR) distributions are its special cases. The GLFR distribution has decreasing or unimodal probability density function (pdf) and its hazard rate function (hrf) can have increasing, decreasing, and bathtub-shaped. Unfortunately, the GLFR distribution cannot have unimodal hrf.

 Recently,  many studies have been done on GLFR distribution, and some authors have extended it: the generalized linear exponential \citep{Ma-al-10}, beta-linear failure rate \citep{ja-ma-2012}, Kumaraswamy-GLFR \citep{elbatal-13}, modified-GLFR \citep{jamkhaneh-14}, McDonald-GLFR \citep{el-me-ma-14}, Poisson-GLFR \citep{co-or-le-15}, GLFR-geometric \citep{na-sh-re-14},  and GLFR-power series \citep{al-sh-14} are some univariate extension of GLFR distribution. \nocite{ma-ja-15}

\cite{ku-gu-11}
proposed an extension of GE distribution that is a very flexible family of distribution. It is positively skewed, and has increasing, decreasing, unimodal and bathtub shaped hrf’s. It is included GE, exponential, generalized Pareto
\citep{jo-ko-ba-95-2},
and Pareto distributions.

\cite{co-ha-or-pa-12} introduced a five-parameter called the McDonald extended exponential distribution  as a generalization of extended generalized exponential (EGE).   In this paper, we introduce a new four-parameter distribution that contains the EGE distribution as its special case. In addition, this class of distribution extends the three-parameter GLFR distribution.  Therefore, it is called extended  generalized linear failure rate (EGLFR) distribution. The hrf of this new distribution is increasing, decreasing, bathtub, and unimodal. In addition, we will show that the new distribution has been fit better than the EGE distribution and other competing distributions to analyzing the lifetime data.

The paper is organized as follows. In Section \ref{sec.class}, we introduce a new class of distributions. Some statistical properties such as moments, quantiles, hrf, and reversible hazard rate are provided in Section \ref{sec.pro}. The maximum likelihood estimation (MLE) of the parameters is obtained in Section \ref{sec.mle}. An application of the EGLFR distribution using a real data set is presented in Section \ref{sec.ex}.

\section{A new class}
\label{sec.class}
 In this section, we introduce a new class of distributions as extension of GLFR distribution.  Also, some properties of the pdf and hrf of this distribution  are given here.

\begin{theorem}
For given  $\alpha>0$, $\beta\in \mathbb{R} $, $a\geq 0$, and $b\geq 0$ (with $a+b>0$), consider the function
\begin{equation}\label{eq.FEGL}
F(x)=\left\{ \begin{array}{ll}
{\left(1-{\left(1-\beta  (ax+\frac{b}{2}x^2)\right)}^{{1}/{\beta }}\right)}^{\alpha } &   {\rm if}\ \ \  \beta \ne 0 \\
{\left(1-{\rm e}^{-(ax+\frac{b}{2}x^2)}\right)}^{\alpha } &                                      {\rm if}\ \  \  \beta =0. \end{array}
\right.
\end{equation}
\noindent i. If $\beta \le 0$,  then $F(x)$ is a cumulative distribution function (cdf) on  $(0,\infty )$.

\noindent ii. If $\beta > 0$ and $b\ne 0$, then $F(x)$ is a cdf on  $(0, \psi )$ where  $\psi =\frac{1}{b}\sqrt{a^2+\frac{2b}{\beta }}-\frac{a}{b}$.

\noindent iii. If $\beta > 0$ and $b= 0$, then $F(x)$ is a cdf on  $(0, \psi )$ where $\psi =\frac{1}{a\beta}$.
\end{theorem}

\begin{proof}
When $\beta=0$, proof is obvious. We consider $\beta\neq0$. Without loss of generality, we consider $\alpha=1$. Therefore,
$\lim_{x\rightarrow 0} F(x) =0$ and $F'(x)=(a+bx)(1-\beta(ax+\frac{b}{2}x^2))^{\frac{1}{\beta}-1}$.

\noindent i. When $\beta<0$, $\lim_{x\rightarrow \infty} F(x)=1$ and $F'(x)>0$  for  $0<x<\infty$.

\noindent ii. When $\beta>0$, $\lim_{x\rightarrow \psi} F(x)=1$  and $F'(x)>0$  for $0<x<\psi$.

\noindent iii. This part is similar to part ii.\\
Therefore, proof is completed.
\end{proof}
%

When the function $F$ in \eqref{eq.FEGL} is a cdf, it is said EGLFR distribution with parameters $\alpha$, $\beta$, $a$ and $b$, and will be denoted by ${\rm EGLFR}(\alpha,\beta,a,b)$.  The pdf of this new class of distributions is
\begin{equation}\label{eq.fEGL}
f\left(x;\alpha ,\beta ,a,b\right)=\left\{ \begin{array}{ll}
\alpha \left(a+bx\right){\left(1-\beta  z\right)}^{\frac{1}{\beta }-1}{\left(1-{\left(1-\beta  z\right)}^{{1}/{\beta }}\right)}^{\alpha -1}
& {\rm if}\ \ \ \ \beta \ne 0 \\
\alpha \left(a+bx\right){\rm e}^{-z}{\left(1-{\rm e}^{-z}\right)}^{\alpha -1}
& {\rm if}\ \ \ \  \beta =0, \end{array}
\right.
\end{equation}
where $z=ax+\frac{b}{2}x^2$. The plots for pdf of EGLFR distribution are given in Figure \ref{fig.den}, for some different values of parameters.

\begin{figure}[ht]
\centering
\includegraphics[width=7.7cm,height=7.7cm]{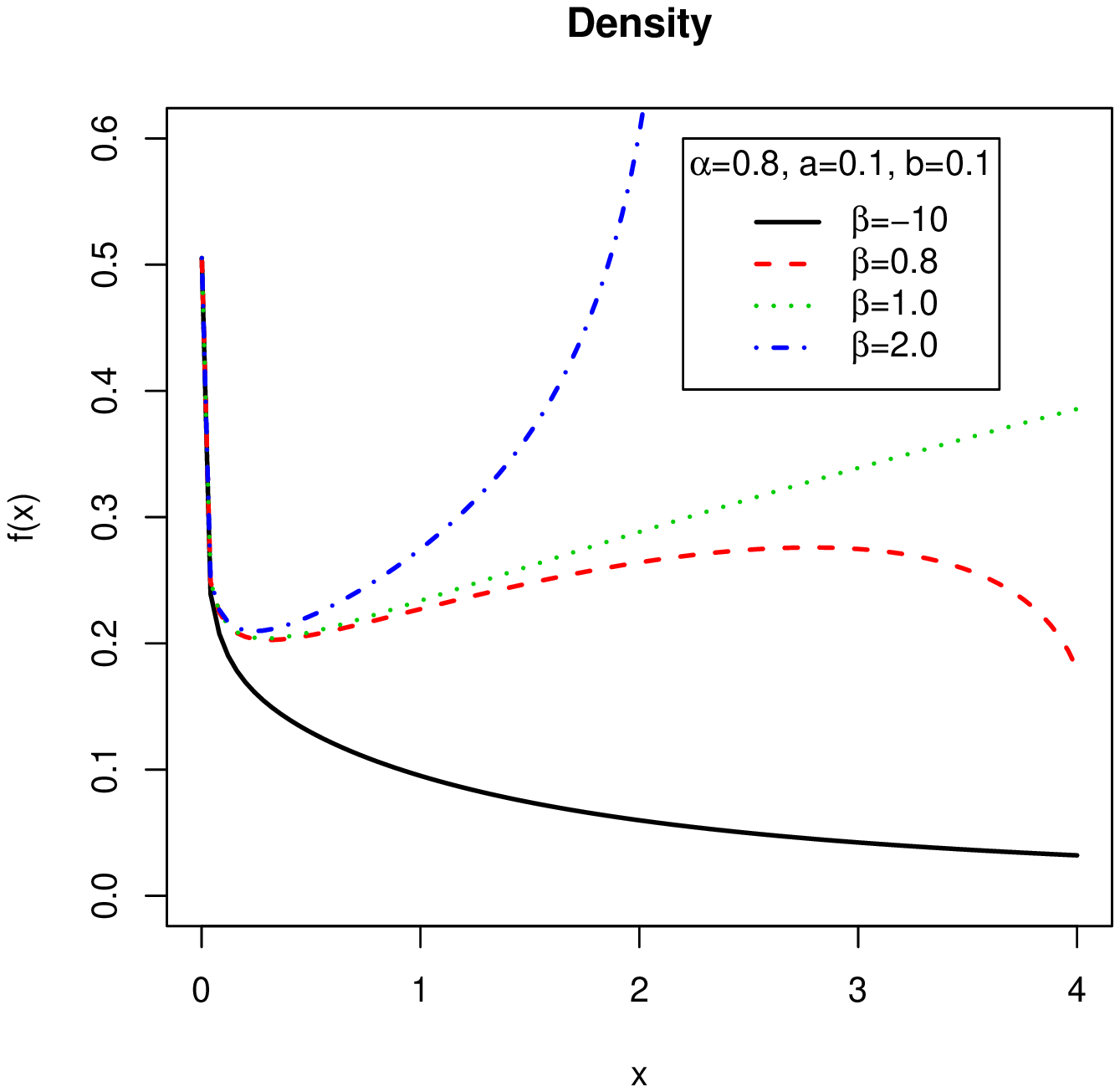}
\includegraphics[width=7.7cm,height=7.7cm]{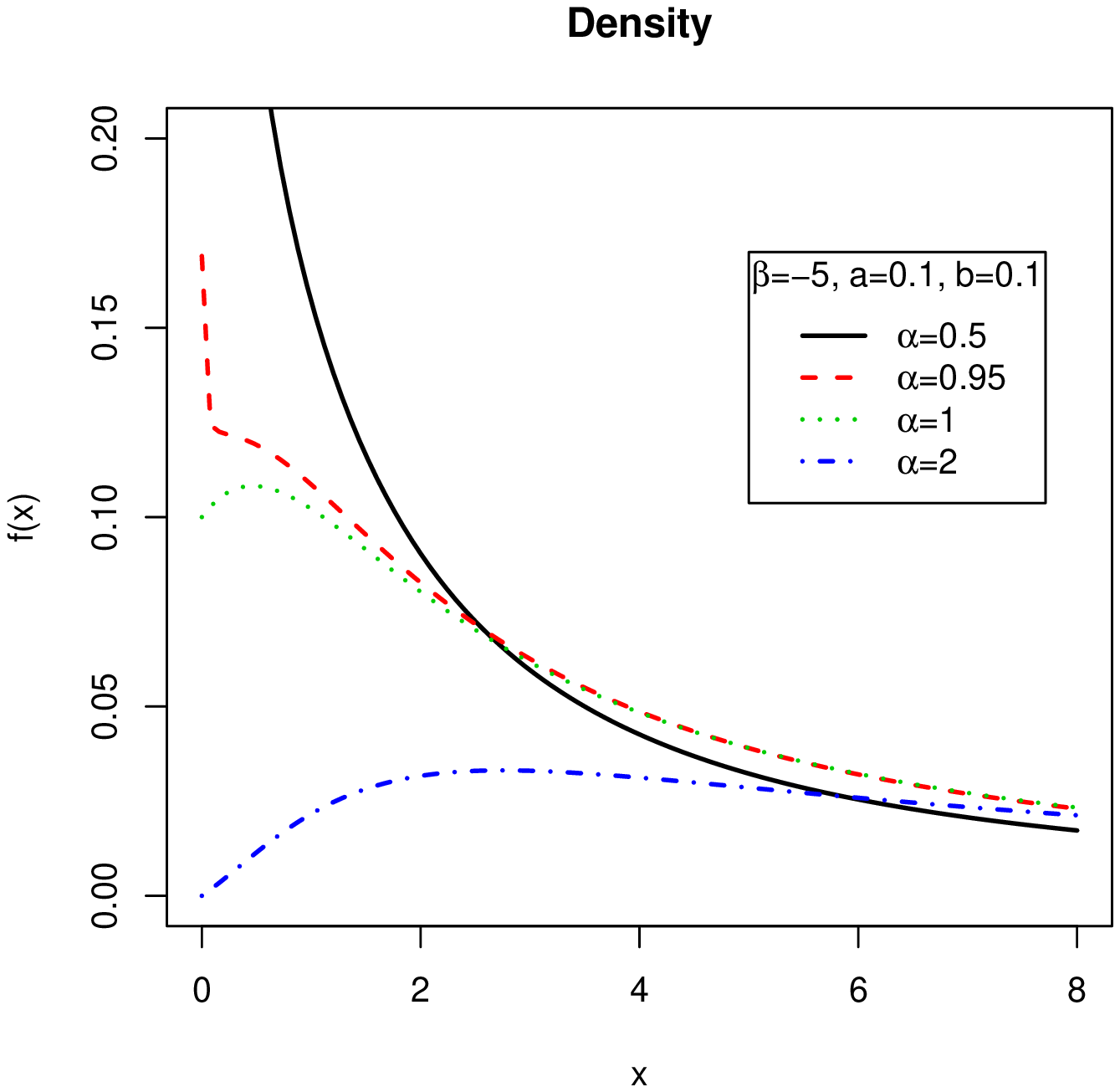}
\vspace{-0.9cm}
\caption{The plot of pdf for some different values of parameters. }\label{fig.den}
\end{figure}

If $\beta =0$, then this distribution reduces to the GLFR distribution which is introduced by
\cite{sa-ku-09}.
The GLFR includes the GE (if $b=0$), exponential (if $b=0$, $\alpha =1$), LFR (if $\alpha =1$), generalized Rayleigh (if $a=0$) and Rayleigh (if $a=0$, $\alpha =1$) distributions as its special sub-models.

If $b=0$ then EGLFR distribution reduces to the EGE distribution introduced by
\cite{ku-gu-11}.
The EGE distribution includes GE (if $\beta =0$), exponential (if $\beta =0$, $\alpha =1$), generalized Pareto
(if $\alpha =1$), and Pareto distributions (if $\alpha =1$, $\beta <0$).

 If  $a=0$, we have an extension of
two-parameter Burr X distribution which is introduced by
\cite{su-pa-05}
 and also is known as
generalized Rayleigh distribution
\citep{ku-ra-05}.
 Therefore,  it is called extended generalized Rayleigh (EGR) distribution.

\begin{prop}\label{prop.1}
If $X$ has a EGLFR distribution with the pdf in \eqref{eq.fEGL}, then $G\left(X\right)=aX+\frac{b}{2}X^2$ has a EGE distribution with parameters $\alpha $, $\beta $, and $1$ or equivalently, $Y^*=X+\frac{b}{2a}X^2$ has a EGE distribution with parameters $\alpha $, $\beta $, and $a$.
\end{prop}

\begin{prop} \label{prop.elfr} Suppose that $X$ has the following cdf
\[F_1\left(x\right)=\left\{ \begin{array}{ll}
1-{\left(1-\beta  z\right)}^{{1}/{\beta }} & {\rm if }\ \ \ \ \beta \ne 0 \\
1-{\rm e}^{-z} & {\rm if}\ \ \ \ \beta =0, \end{array}
\right.\]
where $z=ax+\frac{b}{2}x^2$. We say that $X$ has extended linear failure rate (ELFR) distribution. The maximum of a random sample with size $n$ from the ELFR distribution has ${\rm EGLFR}(n,\beta ,a,b)$. Therefore, the ${\rm EGLFR}(n,\beta ,a,b)$  provides the cdf of a parallel system when each component has the ELFR distribution.
\end{prop}

\begin{theorem}\label{eq.flim}
Let $f(x)$ be the pdf of the EGLFR distribution. The limiting behavior of $f(x)$ for different values of its parameters is given bellow:
\[{\mathop{\lim }_{x\to 0^+} f\left(x\right) }=\left\{ \begin{array}{ll}
0 \ \ & {\rm if\ \  }\alpha >1 \\
a \ \ & {\rm if}\ \  \alpha =1 \\
\infty \ \ & {\rm if}\ \  \alpha <1 \end{array}
\right.{\rm \ \ \ and\ \ \ \ }{\mathop{\lim }_{x\to c^-} f\left(x\right) }=\left\{ \begin{array}{ll}
\infty \ \ & {\rm if\ \  }\beta \geq 1 \\
0 \ \ & {\rm if}\ \  \beta <1, \end{array}
\right.\]
where $c=\psi $ for $\beta >0$, and $c=\infty $ for $\beta <0$.
\end{theorem}
\begin{proof}
The proof is obvious.
\end{proof}

\begin{theorem}
Let $f(x)$ be the pdf of the EGLFR distribution. Then,

\noindent i. the mode of $f(x)$ is obtained from the solution of the following nonlinear equation when $\alpha\geq 1$, $\beta<1$ and $\beta \neq 0$:
$$\frac{b}{a+bx}+\frac{(\beta-1)(a+bx)}{1-\beta  (ax+\frac{b}{2}x^2)}+(\alpha -1)
\frac{(a+bx)\left(1-\beta (ax+\frac{b}{2}x^2)\right)^{\frac{1}{\beta}-1}}{1-{\left(1-\beta  (ax+\frac{b}{2}x^2)\right)^{{1}/{\beta }}}}=0,$$

\noindent ii. the mode of $f(x)$ is obtained from the solution of the following nonlinear equation when $\alpha\geq 1$ and $\beta = 0$:
$$\frac{b}{a+bx}-(a+bx)+(\alpha -1)\frac{a+bx}{{\rm e}^{ax+\frac{b}{2}x^2}-1}=0,$$

\noindent iii.  the mode of $f(x)$ is 0 when $\alpha<1$ and $\beta<1$,

\noindent iv.  $f(x)$ has two modes at 0 and  $\psi$ when $\alpha<1$ and $\beta\geq1$.
\end{theorem}

\begin{proof}
Using Theorem \ref{eq.flim},  the proof is obvious.
\end{proof}

\begin{theorem}
Let $F$ be the pdf of the EGLFR distribution. Then, $F$ is a heavy-tailed distribution when $\alpha\geq1$ and $\beta<0$.
\end{theorem}

\begin{proof}
 Consider $z=ax+\frac{b}{2}x^2$ and $s=-\frac{1}{\beta}$. Therefore,
\[\bar{F}\left(x\right)=1-F(x)=1-{\left(1-{\left(1+\ \frac{z}{s}\right)}^{-s}\right)}^{\alpha }.\]
Without loss of generality take $\alpha=1$. Clearly, we have
$\bar{F}\left(x\right)\sim {\left(\frac{s}{z}\right)}^s$, as $ x\to \infty$. So
\[{\mathop{\lim }_{x\to \infty } {\left(\frac{s}{z}\right)}^s{\rm e}^{\lambda x}=\infty \ \ \ \ {\rm for}\ \forall \lambda >0,\ }\]
and $F$ is a heavy tailed distribution \citep[see][Theorem 2.6]{fo-ko-za-11}.

\end{proof}

The hrf of the EGLFR distribution is
\begin{equation}
h\left(x\right)=\left\{ \begin{array}{ll}
\frac{\alpha \left(a+bx\right){\left(1-\beta  z\right)}^{\frac{1}{\beta }-1}{\left(1-{\left(1-\beta  z\right)}^{{1}/{\beta }}\right)}^{\alpha -1}}{1-{\left(1-{\left(1-\beta  z\right)}^{{1}/{\beta }}\right)}^{\alpha }} &
{\rm if}\ \ \ \  \beta \ne 0 \\
\frac{\alpha \left(a+bx\right){\rm e}^{-z}{\left(1-{\rm e}^{-z}\right)}^{\alpha -1}}{1-{\left(1-{\rm e}^{-z}\right)}^{\alpha }} &
{\rm if}\ \ \ \ \beta =0, \end{array}
\right.
\end{equation}
where $z=ax+\frac{b}{2}x^2$. For some cases of parameters, the plot for hazard of EGLFR distribution are given in Figure \ref{fig.haz}.

\begin{figure}[ht]
\centering
\includegraphics[width=7.7cm,height=7.7cm]{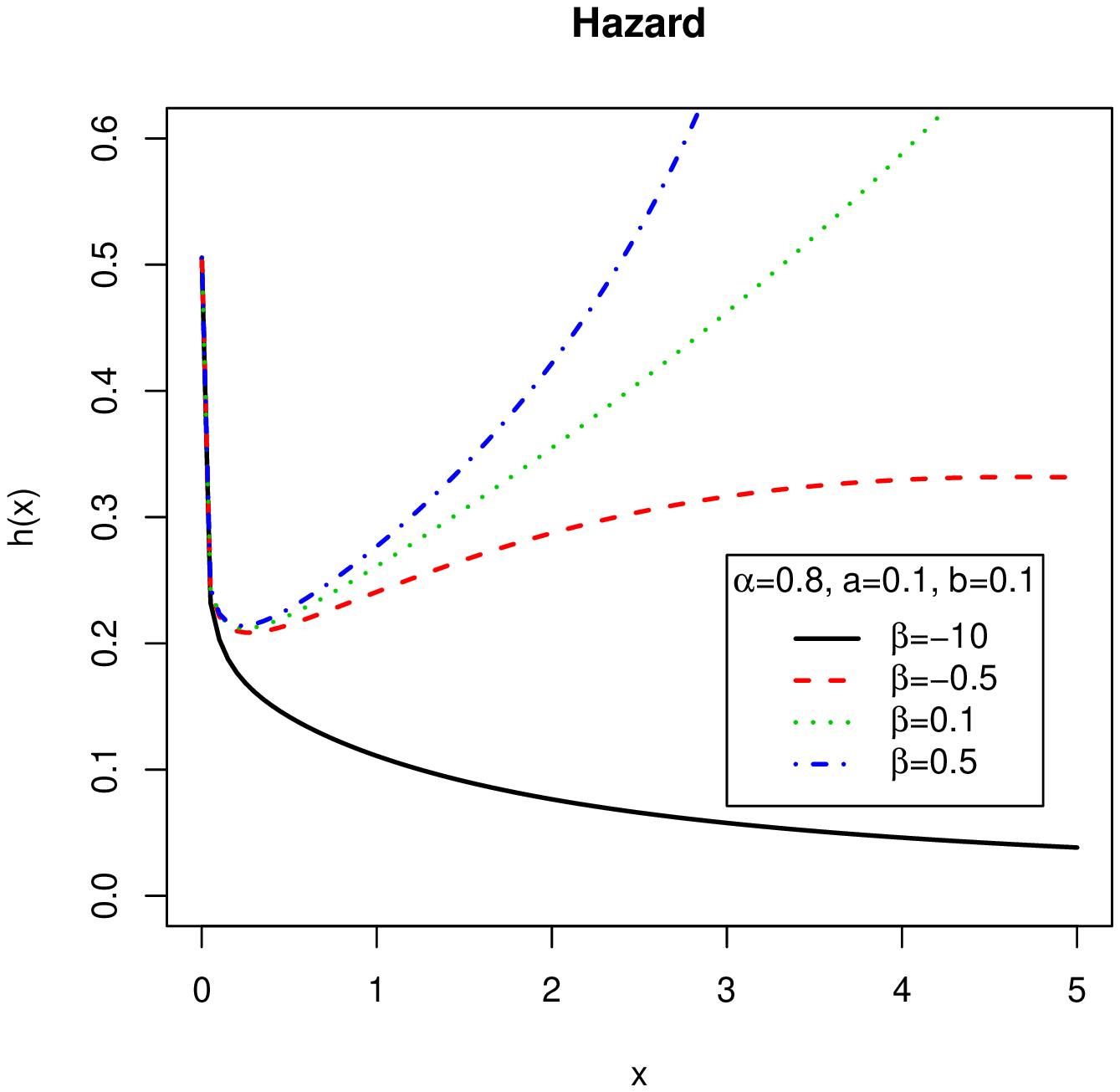}
\includegraphics[width=7.7cm,height=7.7cm]{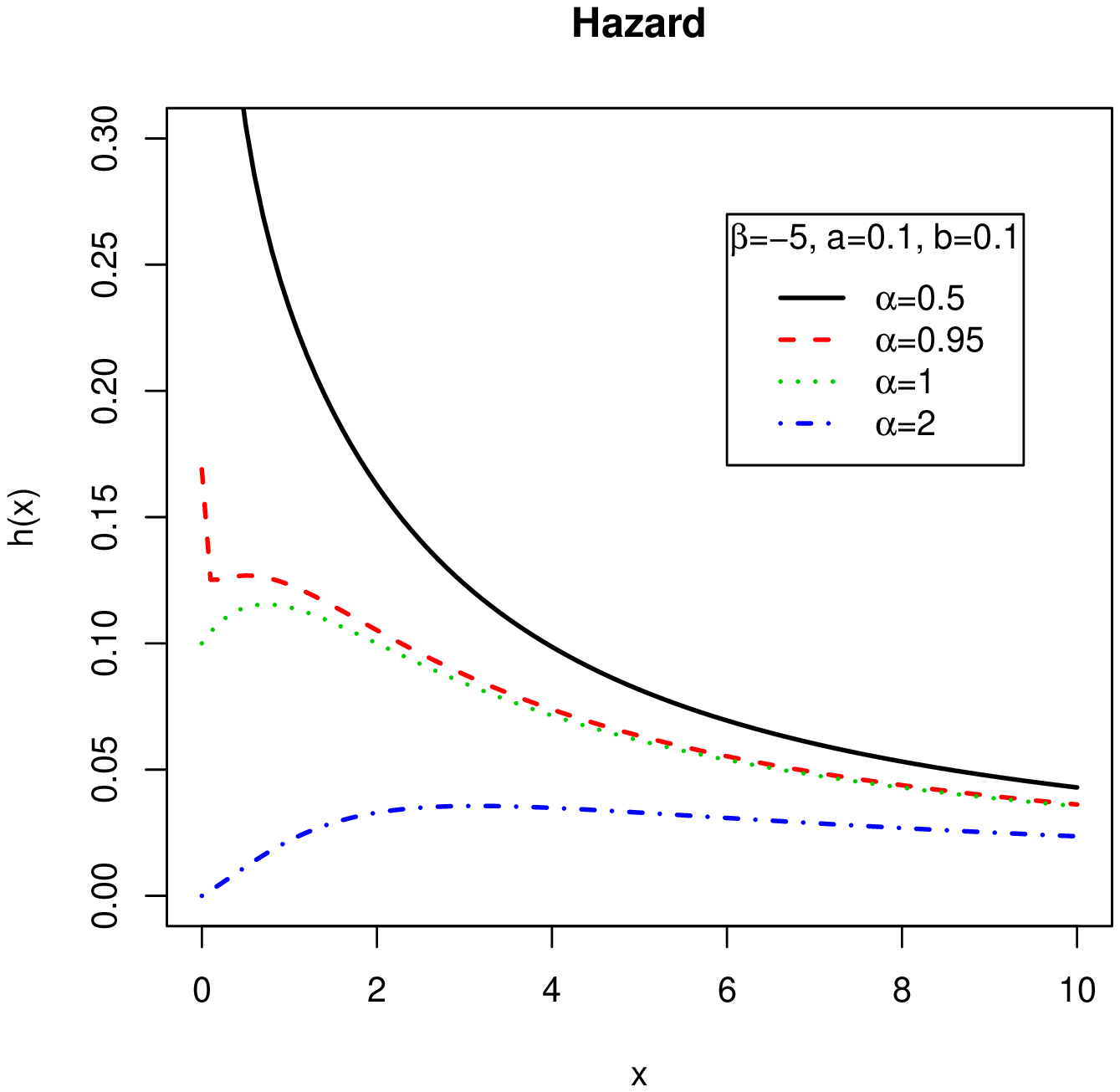}
\vspace{-0.9cm}
\caption{The plot of hrf for some different values of parameters. }\label{fig.haz}
\end{figure}

\begin{theorem} Let $h(x)$ be the hrf of the EGLFR distribution. The limiting behavior of $h(x)$ for different values of its parameters is given bellow:
\[{\mathop{\lim }_{x\to 0^+} h\left(x\right) }=\left\{ \begin{array}{ll}
0 \ \ & {\rm if\ \  }\alpha >1 \\
a \ \ &  {\rm if}\ \  \alpha =1 \\
\infty \ \ & {\rm if}\  \ \alpha <1 \end{array}
\right.{\rm \ \ \ and\ \ \ }{\mathop{\lim }_{x\to c^-} h\left(x\right) }=\left\{ \begin{array}{ll}
\infty \ \ & {\rm if\ \  }\beta >0 \\
0 \ \ & {\rm if}\ \  \beta <0, \end{array}
\right.\]
where $c=\psi $ for $\beta >0$, and $c=\infty $ for $\beta <0$.
\end{theorem}
\begin{proof}
The proof is obvious.
\end{proof}

\begin{theorem}
Let $h(x)$ be the hrf of the EGLFR distribution. Then $h(x)$ is increasing function for $\beta \ge 1$ and $\alpha \ge 1$,
\end{theorem}
\begin{proof}
For $b=0$, we have the EGE distribution and proof is given by
\cite{ku-gu-11}.
Here, we consider $b>0$.
Let $z=ax+\frac{b}{2}x^2=\frac{b}{2}{\left(x+\frac{a}{b}\right)}^2-\frac{a^2}{2b}$. Then, $x=\frac{1}{b}\sqrt{2bz+a^2}-\frac{a}{b}$, and for $\beta \ne 0$
\[h\left(z\right)=\frac{\alpha \sqrt{2bz+a^2}{\left(1-\beta  z\right)}^{\frac{1}{\beta }-1}{\left(1-{\left(1-\beta  z\right)}^{{1}/{\beta }}\right)}^{\alpha -1}}{1-{\left(1-{\left(1-\beta  z\right)}^{{1}/{\beta }}\right)}^{\alpha }}.\]
With $u\left(z\right)={\log  \left(h\left(z\right)\right) }$, we have
\begin{eqnarray*}
\frac{\partial }{\partial z}u\left(z\right)&=&\frac{b}{2bz+a^2}+\left(\beta -1\right)\frac{1}{1-\beta z}+\left(\alpha -1\right)\frac{{\left(1-\beta z\right)}^{\frac{1}{\beta }- 1}}{1-{\left(1-\beta z\right)}^{\frac{1}{\beta }}}\\
&&+\frac{\alpha }{\beta }\frac{{\left(1-{\left(1-\beta  z\right)}^{{1}/{\beta }}\right)}^{\alpha -1}{\left(1-\beta z\right)}^{{1}/{\beta }-1}}{1-{\left(1-{\left(1-\beta z\right)}^{{1}/{\beta }}\right)}^{\alpha }}.
\end{eqnarray*}
If $\beta \ge 1$ and $\alpha \ge 1$, then the hrf is an increasing function.
\end{proof}

The reversible hazard function of EGLFR distribution is
\[r\left(x;\alpha ,\beta ,a,b\right)=\frac{f\left(x;\alpha ,\beta ,a,b\right)}{F\left(x;\alpha ,\beta ,a,b\right)}=\frac{\alpha \left(a+bx\right){\left(1-\beta (ax+\frac{b}{2}x^2)\right)}^{\frac{1}{\beta }-1}}{\left(1-{\left(1-\beta  (ax+\frac{b}{2}x^2)\right)}^{{1}/{\beta }}\right)}.\]
It is clear that $r\left(x;\alpha ,\beta ,a,b\right)=\alpha  r\left(x;1,\beta ,a,b\right).$

\section{ Properties of EGLFR distribution}
\label{sec.pro}
We provide some of statistical properties of introduced distribution such as moments and quantiles. For more properties, reader can see
\cite{sa-ku-09}
and
\cite{ku-gu-11}
for $\beta =0$ and $b=0$, respectively.

\subsection{ Quantiles of EGLFR distribution}

The quantile function of EGLFR distribution is
\[Q\left(u\right)=\left\{ \begin{array}{ll}
\frac{1}{b}\sqrt{a^2+\frac{2b}{\beta }\left[1-{\left(1-\sqrt[{\alpha }]{u}\right)}^{\beta }\right]}-\frac{a}{b} \ \ & {\rm if\ \ \ } \beta \ne 0, b>0 \\
\frac{1}{b}\sqrt{a^2-2b{\log  \left(1-\sqrt[{\alpha }]{u}\right) }}-\frac{a}{b} \ \ & {\rm if}\ \ \  \beta =0, b>0 \\
\frac{1}{\beta a}\left(1-{\left(1-\sqrt[{\alpha }]{u}\right)}^{\beta }\right)\ \ & {\rm if\ \ \ } \beta \ne 0, b=0 \\
-\frac{1}{a}{\log  \left(1-\sqrt[{\alpha }]{u}\right)} \ \ & {\rm if\ \ \ } \beta =0, b=0. \end{array}
\right.\]

The median of the EGLFR distribution can be obtain by letting $u=0.5$. We can use this function for generating data form EGLFR distribution by generating data from a uniform distribution. For checking, we have generated a random sample from ${\rm EGLFR}(0.8,2,0.5,0.1)$. The histogram of this random sample and pdf of EGLFR are plotted in Figure \ref{fig.sim} (left). In addition, the empirical cdf of this sample and cdf of EGLFR are plotted in Figure \ref{fig.sim} (right).

\begin{figure}[ht]
\centering
\includegraphics[width=7.7cm,height=7.7cm]{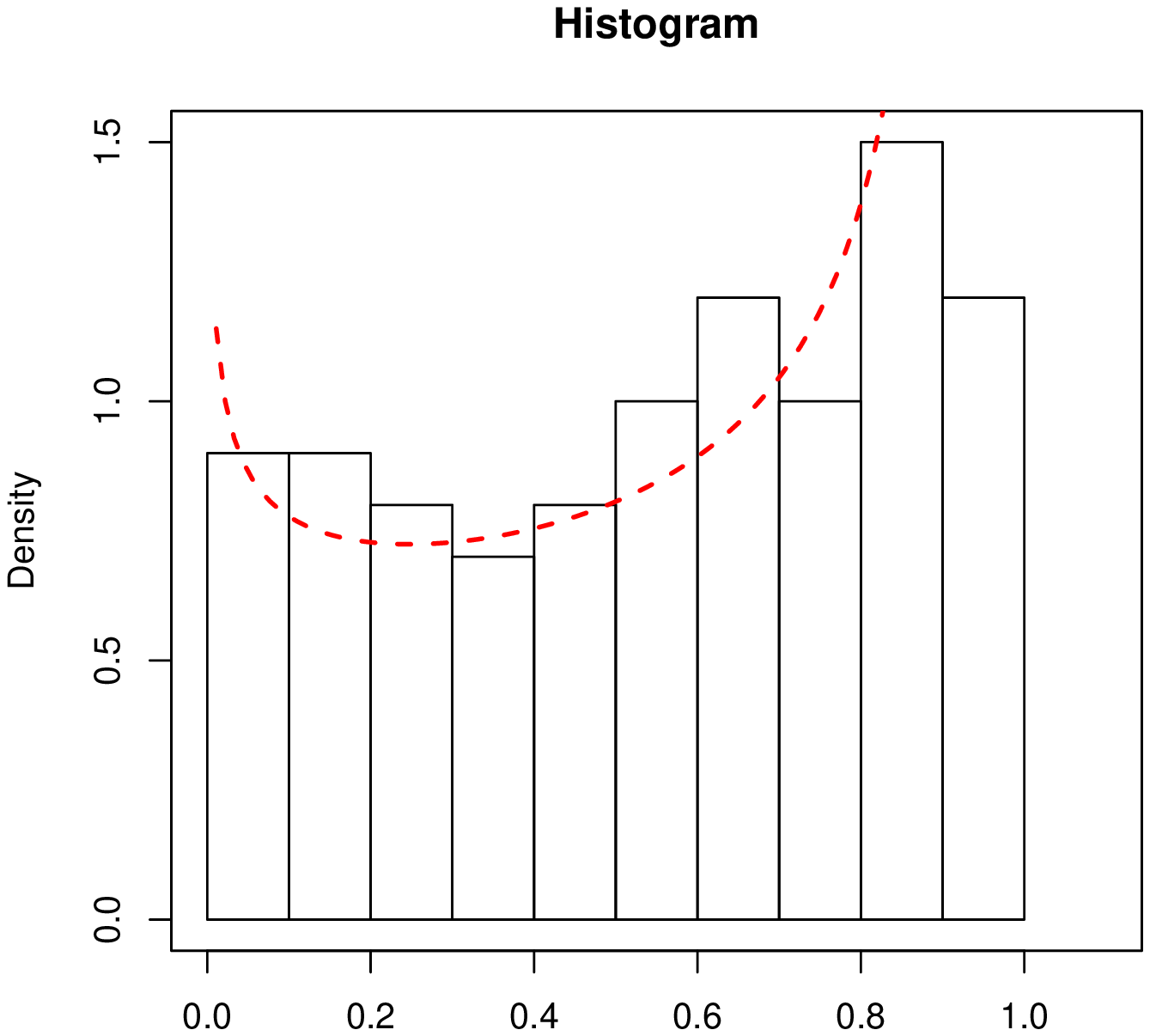}
\includegraphics[width=7.7cm,height=7.7cm]{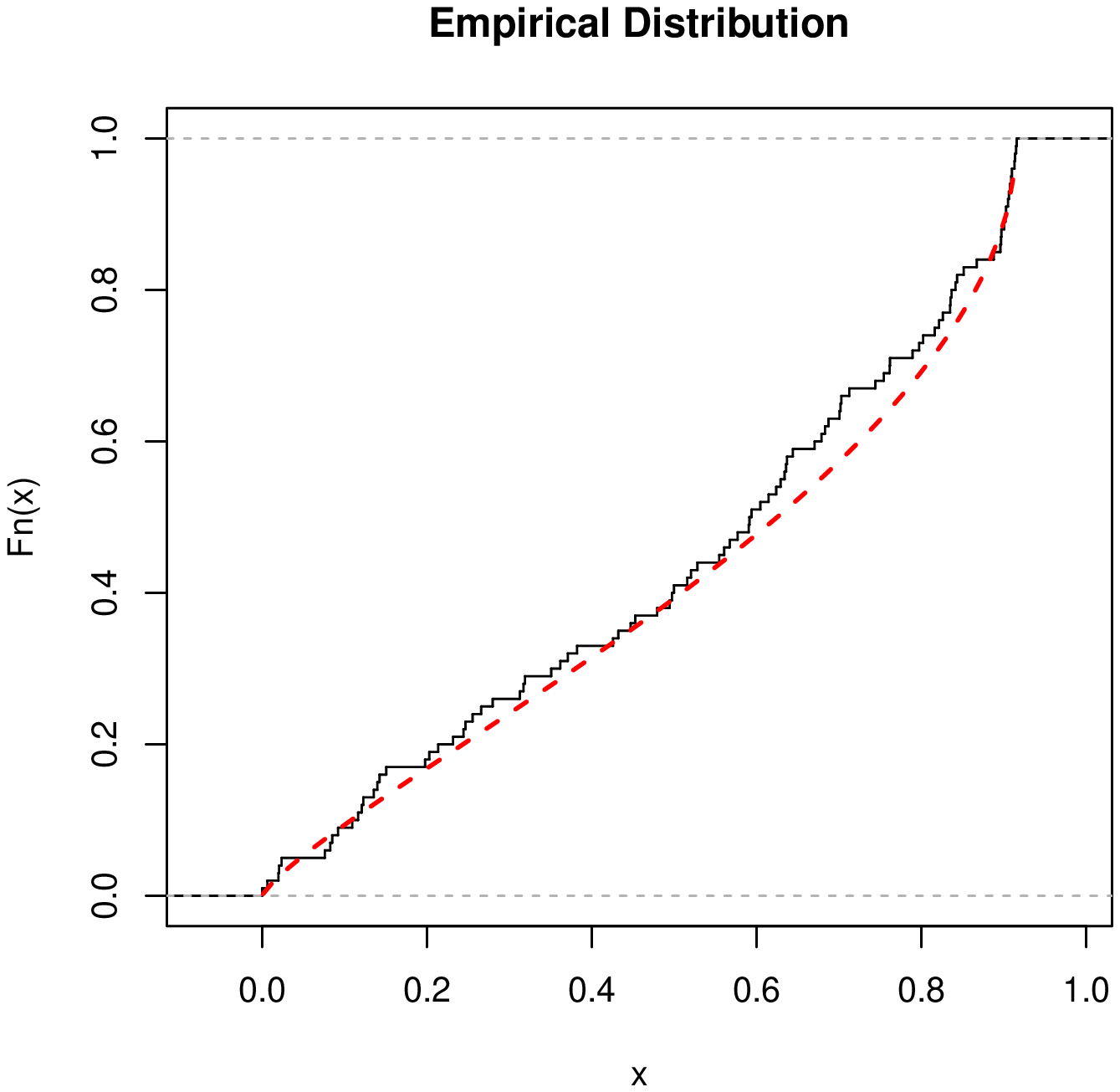}
\vspace{-0.9cm}
\caption{The histogram of the simulated data set and   the  pdf of EGLFR distribution (left); the empirical cdf of the simulated data set and the cdf of EGLFR distribution (right). }\label{fig.sim}
\end{figure}

\subsection{ Skewness and kurtosis of EGLFR distribution}

The Bowley skewness
\citep[see][]{ke-ke-62}
(
based on quantiles can be calculated by
\[B=\frac{Q\left(\frac{3}{4}\right)-2Q\left(\frac{1}{2}\right)+Q\left(\frac{1}{4}\right)}{Q\left(\frac{3}{4}\right)-Q\left(\frac{1}{4}\right)},\]
and the Moors kurtosis
\cite[see][]{moors-88}
is defined as
\[M=\frac{Q\left(\frac{7}{8}\right)-Q\left(\frac{5}{8}\right)+Q\left(\frac{3}{8}\right)-Q\left(\frac{1}{8}\right)}{Q\left(\frac{6}{8}\right)-Q\left(\frac{2}{8}\right)}.\]

These measures are less sensitive to outliers and they exist even for distributions without moments. For the standard normal and the classical standard $t$ distributions with 10 degrees of freedom, the Bowley measure is 0. The Moors measure for these distributions is 1.2331 and 1.27705, respectively. In Figure \ref{fig.bow}, we plot Bowley and Moors measures (with $a=b=0.1$) as a function of $\beta $ for fixed values of $\alpha$.

\begin{figure}[ht]
\centering
\includegraphics[width=7.7cm,height=7.7cm]{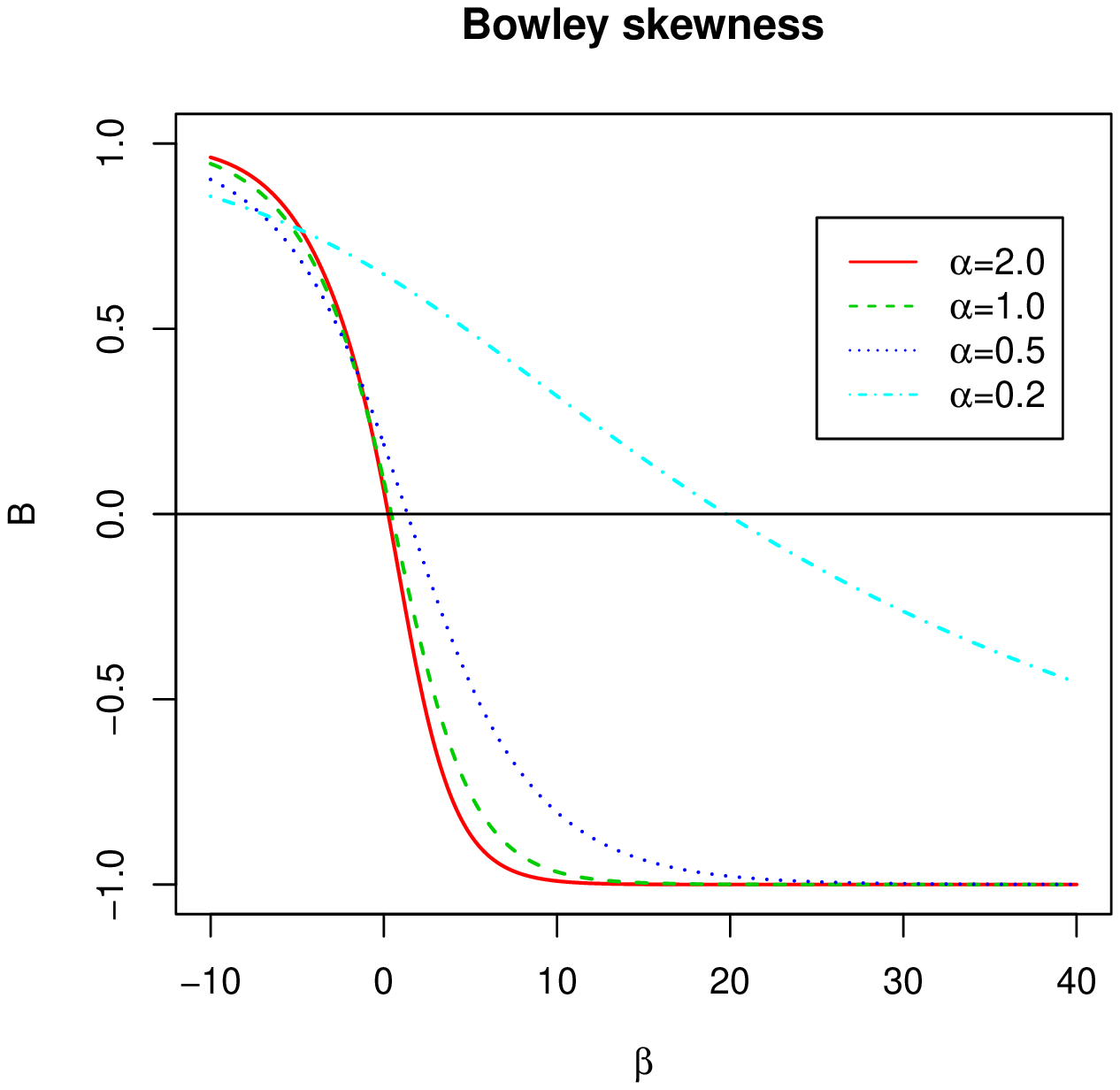}
\includegraphics[width=7.7cm,height=7.7cm]{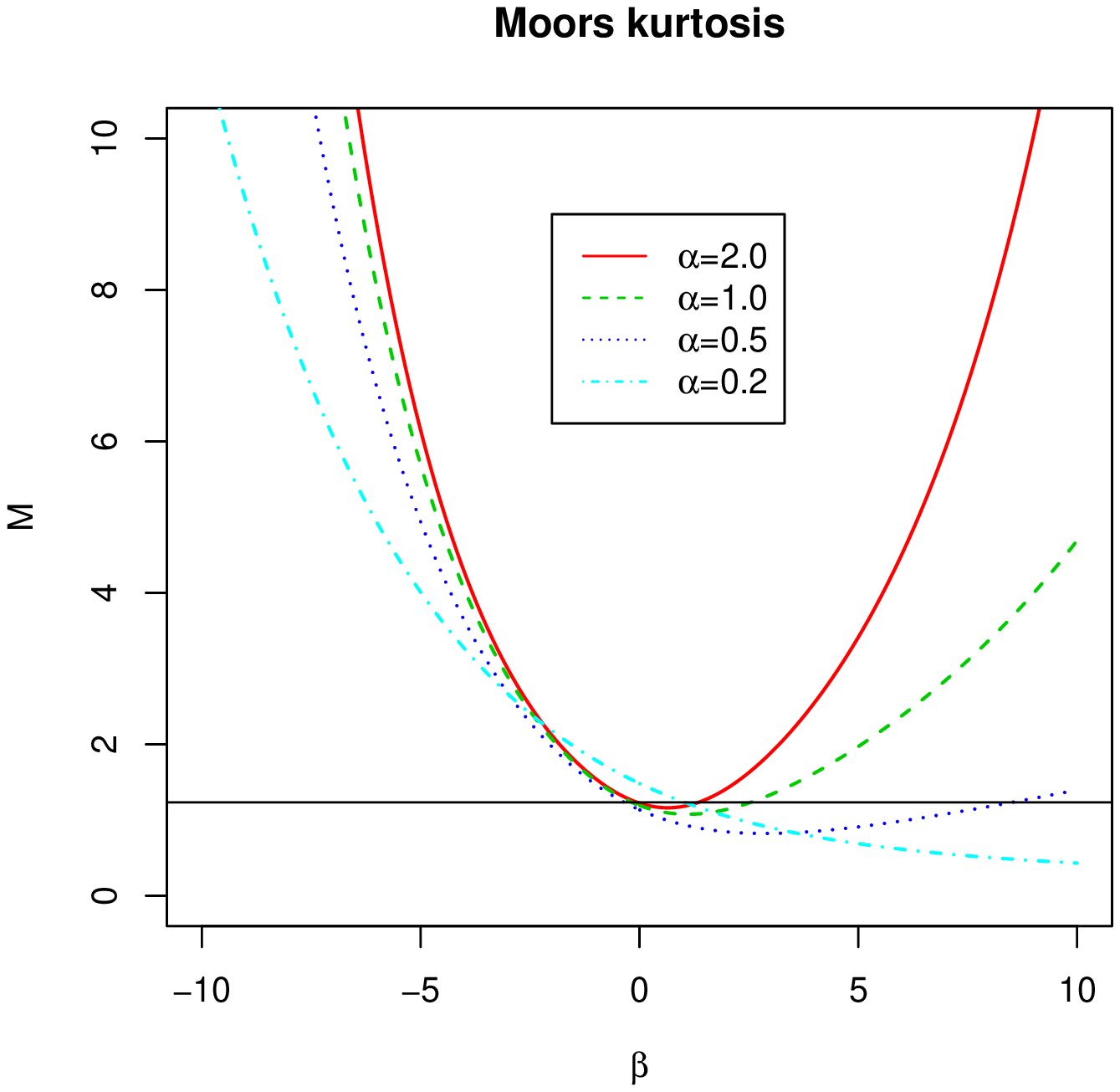}
\vspace{-0.9cm}
\caption{Plot of Bowley measure (left) and Moors measure for some parameters. (right). }\label{fig.bow}
\end{figure}

\subsection{ Moments of EGLFR distribution}

Using the following theorem, we can calculate the $r$-th non-central moment, ${\mu }^{\left(r\right)}$, of EGLFR distribution in a series form.

\begin{theorem} \label{thm.mu} Let $X$ has a ${\rm EGLFR}(\alpha,\beta,a,b)$. Then,  for $\beta >0$ and $b>0$,
\[{\mu }^{\left(r\right)}=E(X^r)=\sum^{\infty }_{n=0}{\sum^{\infty }_{m=0}{\sum^m_{k=0}{\frac{\alpha }{2^k}{\left(-1\right)}^{m+n}}
\Lambda_{[n,m,k]}
\beta^m a^{m-k}b^k{\psi }^j(\frac{a}{j}+\frac{b\psi }{j+1})}},\]
where $j=m+k+r+1$ and  $\Lambda_{[n,m,k]}=\binom{\alpha -1}{n}
\binom{\frac{n+1}{\beta }-1}{m}
\binom{m}{k}$.
\end{theorem}
\begin{proof} Consider $z =ax+\frac{b}{2}x^2$. If $\beta >0$, then support of $X$ is $(0, \psi)$. So,
 $\beta z<1$ and ${(1-\beta z)}^{{1}/{\beta }}<1$. By using binomial series expansion, we have
\begin{eqnarray*}
&&{\left(1-{(1-\beta z)}^{{1}/{\beta }}\right)}^{\alpha -1}=\sum^{\infty }_{n=0}(-1)^n
\binom{\alpha -1}{n}
{(1-\beta z)}^{\frac{n}{\beta }},\\
&&(1-\beta z)^{\frac{n+1}{\beta }-1}=\sum^{\infty }_{m=0}{{\left(-1\right)}^{m}
\binom{\frac{n+1}{\beta }-1}{m}{(\beta z)}^m},\\
&&(a+\frac{b}{2}x)^m=\sum_{k=0}^m \binom{m}{k} a^{m-k} \frac{b^k}{2^k} x^k.
\end{eqnarray*}
Therefore,
\begin{eqnarray*}
{\mu }^{\left(r\right)}&=& \int^{\psi }_0 x^r \alpha \left(a+bx\right)
{\left(1-\beta  z\right)}^{\frac{1}{\beta }-1}{\left(1-{\left(1-\beta  z\right)}^{{1}/{\beta }}\right)}^{\alpha -1}dx\\
&=&\int^{\psi}_0{\alpha x^r\left(a+bx\right)\sum^{\infty }_{n=0}{{\left(-1\right)}^n
\binom{\alpha -1}{n}}{(1-\beta z)}^{\frac{n+1}{\beta }-1}dx}\\
&=& \int^{\psi }_0{\alpha x^r\left(a+bx\right)\sum^{\infty }_{n=0}{\sum^{\infty }_{m=0}{{\left(-1\right)}^{n+m}
\binom{\alpha -1}{n}\binom{\frac{n+1}{\beta }-1}{m}\beta^m(a+\frac{b}{2}x)^m}}dx}\\
&=&\int^{\psi }_0{\alpha x^r\left(a+bx\right)\sum^{\infty }_{n=0}{\sum^{\infty }_{m=0}{\sum^m_{k=0}{{\left(-1\right)}^{n+m}
\binom{\alpha -1}{n}
\binom{\frac{n+1}{\beta }-1}{m}
\binom{m}{k}
{\beta }^ma^{m-k}\frac{b^k}{2^k}x^{m+k}}}}dx}\\
&=&\int^{\psi }_0{\alpha x^r\left(a+bx\right)\sum^{\infty }_{n=0}{\sum^{\infty }_{m=0}{\sum^m_{k=0}{{\left(-1\right)}^{n+m}
\Lambda_{[n,m,k]}
{\beta }^ma^{m-k}\frac{b^k}{2^k}x^{m+k}}}}dx}\\
&=&\sum^{\infty }_{n=0}{\sum^{\infty }_{m=0}{\sum^m_{k=0}{\frac{\alpha }{2^k}{\left(-1\right)}^{m+n}}
\Lambda_{[n,m,k]}
\beta ^ma^{m-k}b^k}}\int^{\psi }_0{x^{m+k+r}\left(a+bx\right)dx}\\
&=&\sum^{\infty }_{n=0}{\sum^{\infty }_{m=0}{\sum^m_{k=0}{\frac{\alpha }{2^k}{\left(-1\right)}^{m+n}}
\Lambda_{[n,m,k]}
\beta^m a^{m-k}b^k{\psi }^j(\frac{a}{j}+\frac{b\psi }{j+1})}},
\end{eqnarray*}
and proof is completed.
\end{proof}

\begin{theorem} \label{thm.mxt} Let $X$ has a ${\rm EGLFR}(\alpha,\beta,a,b)$. Then,  for $\beta >0$ and $b>0$,
\begin{eqnarray*}
M_X(t)=E({\rm e}^{tx})
=\sum^{\infty }_{n =0}\sum^{\infty }_{m  =0}\sum^m_{k=0}\sum^{m+k}_{i=0}{\frac{{\left(-1\right)}^{n+m+i}{\rm e}^{t\psi }\left(m{  +}k\right)!}{t^{i+1}\left(m{  +}k-i\right)!}\frac{\alpha }{{{  2}}^k}{\Lambda }_{\left[n,m,k\right]}{\beta }^ma^{m-k}b^k}{\psi }^{m+k-i}\\
\times
\left[a+\frac{b\left(m{  +}k+1\right)\psi }{\left(m{  +}k+1-i\right)}\right]-\sum^{\infty }_{n{  =0}}{}\sum^{\infty }_{m{  =0}}{}\sum^m_{k{  =0}}
\frac{{\left(-1\right)}^{n+2m+k}\left(a-\frac{b}{t}\right)}{t^{m +k+1}}
\frac{\alpha }{{{  2}}^k}{\Lambda }_{\left[n,m,k\right]}{\beta }^ma^{m -k}b^k.
\end{eqnarray*}
where
 $\Lambda_{[n,m,k]}=\binom{\alpha -1}{n}\binom{\frac{n+1}{\beta }-1}{m}
\binom{m}{k}$.
\end{theorem}

\begin{proof}

Since \citep[see][Section 2.321]{gr-ry-07}
\[\int{x^n{\rm e}^{tx}dx}={\rm e}^{tx}n!\sum^n_{i=0}{\frac{{\left(-1\right)}^i}{t^{i+1}\left(n-i\right)!}}x^{n-i},\]
we have
\begin{eqnarray*}
\int^{\psi }_0x^{m+k}\left(a+bx\right){\rm e}^{tx}dx&=&a\int^{\psi }_0x^{m +k}{\rm e}^{tx}dx+b\int^{\psi }_0{}x^{m{  +}k+1}{\rm e}^{tx}dx\\
&=&
{\rm e}^{t\psi }\left(m{  +}k\right)!\sum^{m{  +}k}_{i=0}{\frac{{\left(-1\right)}^i}{t^{i+1}\left(m+k-i\right)!}}\psi^{m +k-i}\left[a+\frac{b\left(m{  +}k+1\right)\psi }{\left(m+k+1-i\right)}\right]\\
&&
-\frac{{\left(-1\right)}^{m{  +}k}(a-\frac{b}{t})}{t^{m{  +}k+1}}.
\end{eqnarray*}
Similar to Theorem \ref{thm.mu}, we have
$$
M_{X}(t)=\sum\limits_{n=0}^{\infty}\sum\limits_{m=0}^{\infty}\sum\limits_{k=0}^{m}\frac{\alpha}{2^{k}}(-1)^{n+m}\Lambda_{[n,m,k]}\beta^{m}a^{m-k}b^{k}\int_{0}^{\psi} x^{^{m+k}}(a+bx){\rm e}^{tx}dx.
$$
and proof is completed.
\end{proof}

\begin{corollary} For $\beta >0$, $b>0$ and $a=0$, from  Theorem \ref{thm.mu} we have
\[{\mu }^{\left(r\right)}=\sum^{\infty }_{n=0}{\sum^{\infty }_{m=0}{\frac{\alpha b}{\left(2m+r+1\right)}{\left(-1\right)}^{n+m}
\binom{\alpha -1}{n}
\binom{\frac{n+1}{\beta }-1}{m}
{(\frac{2}{b\beta })}^{\frac{r+1}{2}} .}}\]
\end{corollary}

\begin{remark} \label{rem.moment} Using the equation (6) of
\cite{ku-gu-11}
and proposition \ref{prop.1} for $Y=aX+\frac{b}{2}X^2$, we have
\[E\left({\left(1-\beta Y\right)}^k\right)=\frac{\Gamma (\alpha +1)\Gamma (k\beta +1)}{\Gamma (\alpha +k\beta +1)}:={{\rm g}}_k,\ \ \ k\beta +1>0,\ \ \ \ \beta \ne 0. \]
When $a=0$,
\[E(X^{2k})=(-1)^k(\frac{2}{\beta b})^k\left({{\rm g}}_k-1-\sum^{k-1}_{r=1}{{(-1)}^r}
\binom{k}{r}
(\frac{\beta b}{2})^r E(X^{2r})\right),\ \ \ k=2,3,\dots \ .\]
Therefore,
\[E\left(X^2\right)=\frac{2}{b\beta }\left(1-{{\rm g}}_1\right),\ \ \ \ \ \ \ \ E\left(X^4\right)=\frac{4}{b^2{\beta }^2}\left({{\rm g}}_2-2{{\rm g}}_1+1\right).\]
\end{remark}

\begin{theorem} \label{thm.moment}
Let $X$ has a ${\rm EGLFR}(\alpha,\beta,a,b)$. Then,  for $\beta <0$ and $\alpha\geq1$, the k-th moment of $X$ does not exist when $k\ge -\frac{2}{\beta }$ .
\end{theorem}

\begin{proof}
Consider $\alpha =1$. Then, the survival function of $X$ is
\[\bar{F}\left(x\right)={\left(1+\frac{z}{s}\right)}^{-s},\]
where $z=ax+\frac{b}{2}x^2$. Since $X$ is a positive random variable, so
\[E\left[X^k\right]=\int^{\infty }_0{P\left(X>x^{{1}/{k}}\right)}dx=\int^{\infty }_0{\overline{F}\left(x^{{1}/{k}}\right)}dx=\int^{\infty }_0{\frac{dx}{{\left(1+\frac{ax^{{1}/{k}}+\frac{b}{2}x^{{2}/{k}}}{s}\right)}^s},}\]
where $s=-\frac{1}{\beta }$. clearly,
$(s+ax^{{1}/{k}}+\frac{b}{2}x^{{2}/{k}})/{s}\sim \frac{b}{2s}x^{{2}/{k }}$,  as $x\to \infty$. The integral $\int^{\infty }_0{x^{{-2s}/{k}}}dx$ diverges when $k\ge 2s$, and therefore, $E\left[X^k\right]$ does not exists.

Consider $F_{\alpha }\left(x\right)={\left(1-{\left(1+\frac{z}{s}\right)}^{-s}\right)}^{\alpha }$. It can be easily seen that if ${\alpha }_1<{\alpha }_2$ then $F_{{\alpha }_1}\left(x\right)>F_{{\alpha }_2}\left(x\right)$ and for the survival function $\bar{F}\left(x\right)$ we have ${\bar{F}}_{{\alpha }_1}\left(x\right)<{\bar{F}}_{{\alpha }_2}\left(x\right)$. So, (at least) the divergence condition $k\ge 2s$, satisfies for $\forall \alpha \ge 1.$
\end{proof}

Based on Remark \ref{rem.moment}, if $k$-th moment of EGLFR distribution exists, we  have an expression for it when $k$ is even. But, we did not find an expression for other cases. In Table \ref{tab.mom}, we calculated some moments of EGLFR distribution for $\beta<0$ when they  exist with considering $a=0$ and $b=1$.

\begin{table}
\begin{center}
\caption{Some moments of EGLFR distribution for $\beta<0$.}\label{tab.mom}
\begin{tabular}{|c|c|cccccc|} \hline
 &  &   \multicolumn{6}{|c|}{$\beta $} \\ \hline
$\alpha $ & $k$ & -1.90 & -1.00 & -0.75 & -0.50 & -0.25 & -0.10 \\ \hline
1 & 1 & 20.195 & 2.221 & 1.829 & 1.571 & 1.388 & 1.303 \\
 & 2 & --- & --- & 8.000 & 4.000 & 2.667 & 2.222 \\
 & 3 & --- & --- & --- & 18.849 & 6.664 & 4.598 \\
 & 4 & --- & --- & --- & --- & 21.333 & 11.111 \\
 & 5 & --- & --- & --- & --- & 88.857 & 30.656 \\ \hline
2 & 1 & 38.886 & 3.332 & 2.617 & 2.159 & 1.846 & 1.702 \\
 & 2 & --- & --- & 14.400 & 6.667 & 4.190 & 3.391 \\
 & 3 & --- & --- & --- & 35.342 & 11.610 & 7.732 \\
 & 4 & --- & --- & --- & --- & 39.619 & 19.883 \\
 & 5 & --- & --- & --- & --- & 171.467 & 57.127 \\ \hline
\end{tabular}
\end{center}
\end{table}

\subsection{Entropy}

The entropy of random variable is defined in terms of its probability distribution and can be shown to be a good measure of randomness or uncertainty. The Shannon's entropy of a continuous random variable $Y$ with pdf $f(y)$ is defined by
\cite{shan-48}
as
\begin{equation*}
H_{Sh} (f)= - E_f [\log f(Y)]= - \int_{0}^{\infty} f(y) \log f(y) dy.
\end{equation*}
Since the pdf of EGLFR distribution can be written as
\begin{eqnarray}
f(x)=\alpha f_{1}(x)[F_{1}(x)]^{\alpha-1}
\end{eqnarray}
where $f_{1}$ and $F_{1}$ are pdf and cdf  of ELFR distribution, respectively, given in Proposition \ref{prop.elfr},
the Shannon entropy for EGLFR distribution can be expressed in the form
\begin{eqnarray}\label{eq.HSH}
H_{Sh} (f)= &=&-\ln(\alpha)+\frac{\alpha-1}{\alpha}- E_{f}[\ln f_{1}(X)]\nonumber\\
&=&-\ln(\alpha)+\frac{\alpha-1}{\alpha}-E_{W}[\ln f_{1}(F_{1}^{-1}(W))],
\end{eqnarray}
where $W$ has the beta distribution with parameters $\alpha$ and 1.
The last  term in \eqref{eq.HSH} follows immediately from the conditions in Lemma 1 of
\cite{zo-ba-09}.
The R\'{e}nyi entropy is defined by
\begin{equation*}
H_{\rho} (f)= \frac{1}{1-\rho}  \log (  \int_{-\infty}^{\infty} [f(y)]^{\rho} dy ),
\end{equation*}
where $\rho >0$ and $\rho \neq 1$. The Shannon entropy is derived from $\lim_{\rho \rightarrow 1} H_{\rho} (f)$.
An explicit expression of R\'{e}nyi entropy for EGLFR distribution is obtained as
\begin{eqnarray*}
H_{\rho} (f) &=&\frac{-\rho}{\rho-1}\ln(\alpha)-\frac{1}{\rho-1}\ln B(\rho(\alpha-1)+1,1)-\frac{1}{\rho-1}\ln E_{T}[ f_{1}^{\rho-1}(F_{1}^{-1}(T))],
\end{eqnarray*}
where $T$ has the beta distribution with parameters $\rho(\alpha-1)+1$ and 1.

%

\subsection{Characterization}

Using the following theorem, we can  characterize the EGLFR distribution.

\begin{theorem} The random variable $X$ follows ${\rm EGLFR}(\alpha ,\beta ,a,b)$ if and only if for all real $t>0$, and for all non-negative integer $n$
\[{\delta }^{\left(n\right)}\left(t\right)=U^n\left(t\right)+\frac{n}{\alpha }{\delta }^{\left(n-1\right)}\left(t\right),\]
where ${\delta }^{\left(n\right)}\left(t\right)=E\left(U^n\left(X\right)\left | X<t\right.\right)$ and
$U\left(t\right)=-{\log  \left({red}D(t)\right)}=-{\log  (1-{(1-\beta (at+\frac{b}{2}t^2))}^{{1}/{\beta }})}.$
\end{theorem}

\begin{proof}

\noindent \textit{Necessity condition: }

 Let $F$ and $f$ be the cdf and pdf of ${\rm EGLFR}(\alpha ,\beta ,a,b)$, respectively. Also,  let $d(x)$ be the derivative function of  $D(x)$. It can be shown that
\begin{eqnarray*}
{\delta }^{\left(n\right)}\left(t\right)
&=&\frac{\alpha }{F\left(t\right)}\int^t_0{{\left[-{\log  \left(D(x)\right)}\right]}^nd(x){\left[D(x)\right]}^{\alpha -1}dx}
\\
&=&\frac{\alpha }{F\left(t\right)}\int^{D\left(t\right)}_0{{\left[-{\log  \left(z\right)}\right]}^nz^{\alpha -1}dz}
\\
&=&\frac{1}{F\left(t\right)}\int^{D\left(t\right)}_0{{\left[-{\log  \left(z\right)}\right]}^ndz^{\alpha }}
\\
&=&{\left.\frac{{\left[-{\log  \left(z\right) }\right]}^nz^{\alpha }}{F\left(t\right)}\right|}^{D\left(t\right)}_0+\frac{n}{F\left(t\right)}\int^{D\left(t\right)}_0{{\left[-{\log  \left(z\right) }\right]}^{n-1}z^{\alpha -1}dz}
\\
&=&U^n\left(t\right)+\frac{n}{\alpha }{\delta }^{\left(n-1\right)}\left(t\right).
\end{eqnarray*}

\noindent \textit{Sufficiency condition:}
\cite{sa-ku-09}
showed that ${\delta }^{\left(n\right)}\left(t\right)=U^n\left(t\right)+\frac{n}{\alpha }{\delta }^{\left(n-1\right)}\left(t\right),$ implies that
\[r\left(x\right)=\frac{f\left(x\right)}{F\left(x\right)}=\alpha \frac{f\left(x;1,\beta ,a,b\right)}{F\left(x;1,\beta ,a,b\right)}=\alpha r\left(x;1,\beta ,a,b\right).\]
Therefore, from the uniqueness property of reversible hazard function, we can conclude that $X$ follows a ${\rm EGLFR}(\alpha ,\beta ,a,b)$.
\end{proof}

\section{ Estimation}

\label{sec.mle}
In this section, we discuss the MLE of the parameters of the model when $\beta \ne 0$. Consider $X_1,\dots,$ $X_n$ is a random sample from EGLFR distribution with vector parameter ${\boldsymbol \theta }=(\alpha ,\beta ,a,b)'$. The log-likelihood function based on this random sample is given as
\begin{equation}\label{eq.lik}
\ell ({\boldsymbol \theta })=n\log(\alpha)  +\sum^n_{i=1}{{\log (a+bx_i)}}+(\frac{1}{\beta }-1)\sum^n_{i=1}{{\log (1-\beta z_i) }}+(\alpha -1)\sum^n_{i=1}{{\log  \left(1-{(1-\beta z_i)}^{\frac{1}{\beta }}\right) }},
\end{equation}
where $z_i=ax_i+\frac{b}{2}x^2_i$. The log-likelihood can be maximized either directly or by solving the nonlinear likelihood equations obtained by differentiating \eqref{eq.lik}. The components of the score vector
$U\left({\boldsymbol \theta }\right)={\left(U_{\alpha }\left({\boldsymbol \theta }\right),U_{\beta }\left({\boldsymbol \theta }\right),U_a\left({\boldsymbol \theta }\right),U_b\left({\boldsymbol \theta }\right)\right)}^T$
are given by
\begin{eqnarray*}
U_{\alpha }\left({\boldsymbol \theta }\right)&=&\frac{\partial \ell \left({\boldsymbol \theta }\right)}{\partial\alpha }=\frac{n}{\alpha }+\sum^n_{i=1}{{\log  \left(1-{\left(1-\beta z_i\right)}^{\frac{1}{\beta }}\right) }},\\
U_{\beta }\left({\boldsymbol \theta }\right)&=&\frac{\partial \ell \left({\boldsymbol \theta }\right)}{\partial\beta }=-\frac{1}{{\beta }^2}\sum^n_{i=1}{{\log  \left(1-\beta z_i\right) }}-\left(\frac{1}{\beta }-1\right)\sum^n_{i=1}{\frac{z_i}{1-\beta z_i}}\\
&&+\frac{\alpha -1}{{\beta }^2}\sum^n_{i=1}{\frac{{\left(1-\beta z_i\right)}^{\frac{1}{\beta }-1}}{1-{\left(1-\beta z_i\right)}^{\frac{1}{\beta }}}[{\beta z}_i+\left(1-\beta z_i\right){\log  \left(1-\beta z_i\right) }]},\\
U_a\left({\boldsymbol \theta }\right)&=&\frac{\partial \ell \left({\boldsymbol \theta }\right)}{\partial a}=\sum^n_{i=1}{\frac{1}{a+bx_i}}-\left(\frac{1}{\beta }-1\right)\sum^n_{i=1}{\frac{\beta x_i}{1-\beta z_i}}+\left(\alpha -1\right)\sum^n_{i=1}{\frac{x_i{\left(1-\beta z_i\right)}^{\frac{1}{\beta }-1}}{1-{\left(1-\beta z_i\right)}^{\frac{1}{\beta }}}},\\
U_b\left({\boldsymbol \theta }\right)&=&\frac{\partial \ell \left({\boldsymbol \theta }\right)}{\partial b}=\sum^n_{i=1}{\frac{x_i}{a+bx_i}}-\left(\frac{1}{\beta }-1\right)\sum^n_{i=1}{\frac{\beta x^2_i}{2\left(1-\beta z_i\right)}}+\left(\alpha -1\right)\sum^n_{i=1}{\frac{x^2_i{\left(1-\beta z_i\right)}^{\frac{1}{\beta }-1}}{2\left(1-{\left(1-\beta z_i\right)}^{\frac{1}{\beta }}\right)}}.
\end{eqnarray*}

For given $\beta $, $a$, and $b$, the MLE of parameter $\alpha $ is
\[\hat{\alpha }=-\frac{n}{\sum^n_{i=1}{{\log  \left(1-{\left(1-\beta z_i\right)}^{\frac{1}{\beta }}\right) }}}.\]
By maximizing the profile log-likelihood $\ell \left(\hat{\alpha },\beta ,a,b\right)$ with respect to $\beta ,a,b$, the MLE of these parameters can be obtained.

\subsection{ The regular case}

When $\beta \le 0$, the support of the EGLFR distribution is in interval $\left(0,\infty \right)$, and in this case the asymptotic distribution of the MLE of vector parameter ${\boldsymbol \theta }=\left(\alpha ,\beta ,a,b\right)'$ is multivariate normal distribution as
\[\sqrt{n}(\hat{{\boldsymbol \theta }}-{\boldsymbol \theta }){{\stackrel{\ \ \ d\ \ \ }{\longrightarrow}}}N_4\left({\boldsymbol 0},I^{-1}\right),\]
where $I$ is the Fisher information matrix.

\subsection{The non-regular case}

When $\beta >0$, the support of the EGLFR distribution is in interval $\left(0,\psi \right)$. So, in this case the standard asymptotic normality distribution of the MLE's of parameters does not hold. The asymptotic distributions are proposed by
\cite{ku-gu-11}
 for $b=0$. Here, similar results are given when $b>0$.  The general approach is given by \cite{smith-85}, and is used in literature  \citep[for example][]{ng-lu-hu-du-12}.

Take $G(x)= ax+\frac{b}{2}x^2$. Then $G^{-1}(\frac{1}{\beta })=\psi $, and the pdf of EGLFR distribution can be written as
\begin{equation}\label{eq.fegl2}
f(x;\alpha ,a,b,\psi)=\alpha {\rm g}(x){(1-t)}^{G(\psi)-1}{\left(1-{\left(1-t\right)}^{G\left(\psi \right)}\right)}^{\alpha -1},\ \ \ x\le \psi ,
\end{equation}
where $t=G\left(x\right)/G\left(\psi \right)$, and${\rm g}\left(x\right)$ is the derivative of $G\left(x\right)$ with respect to $x$. So, the log-likelihood function correspond to the pdf in \eqref{eq.fegl2} is
\begin{eqnarray*}
\ell (\alpha,a,b,\psi )&=&n{\log (\alpha )}+\sum^n_{i=1}{{\log \left({\rm g}(x_{(i)})\right)}}+\left(G\left(\psi \right)-1\right)\sum^n_{i=1}{{\log  \left(1-t_{(i)}\right)}}\\
&&
+\left(\alpha -1\right)\sum^n_{i=1}{{\log  \left(1-{\left(1-t_{(i)}\right)}^{G(\psi )}\right) }},
\end{eqnarray*}
where $t_{(i)}=G(x_{(i)})/G\left(\psi \right)$, and $x_{(1)},x_{\left(2\right)},\dots ,x_{\left(n\right)}$ are the ordered statistics of random sample $x_1,x_2,\dots ,x_n$. As mentioned in
\cite{smith-85},
at first, we find the MLE of the threshold parameter of the model. For the EGLFR distribution, the MLE of the threshold parameter $\psi $is $\tilde{\psi }=x_{\left(n\right)}$. Then, the modified log-likelihood function based on the remaining $(n-1)$ observations is
\begin{eqnarray*}
\ell (\alpha ,a,b,\tilde{\psi })&=&(n-1){\log (\alpha)}+\sum^{n-1}_{i=1}{{\log  \left({\rm g}(x_{(i)})\right)
 }}+(G(\tilde{\psi })-1)\sum^{n-1}_{i=1}{{\log  \left(1-{\tilde{t}}_{(i)}\right)}}
\\
&&
+\left(\alpha -1\right)\sum^{n-1}_{i=1}{{\log  \left(1-{(1-{\tilde{t}}_{(i)})}^{G\left(\tilde{\psi }\right)}\right) }},
\end{eqnarray*}
where ${\tilde{t}}_{(i)}=G(x_{\left(i\right)})/G\left(\tilde{\psi }\right)$. For more information about the modified likelihood function, refer to
\cite{smith-85}.
For fixed $a$ and $b$ the modified MLE of parameter $\alpha $ is obtained as
\[\tilde{\alpha }=-\frac{n-1}{\sum^{n-1}_{i=1}{{\log  \left(1-{(1-{\tilde{t}}_{(i)})}^{G\left(\tilde{\psi }\right)}\right)}}}.\]
By maximizing the modified log-likelihood $\ell \left(\tilde{\alpha },a,b,\tilde{\psi }\right)$ with respect to $a$, and $b$, the modified MLE of these parameters can be obtained.

\begin{theorem}
The asymptotic distribution of $\tilde{\psi }$ is
\[n^{1/G(\psi )}(\tilde{\psi }-\psi ){{\stackrel{\  \ \ d\  \ \ }{\longrightarrow}}}-\frac{G\left(\psi \right)}{\left(a+b\psi \right)}V^{1/G\left(\psi\right)},\]
where the random variable $V$ is distributed as exponential distribution with mean $\frac{1}{\alpha }$.
\end{theorem}

\begin{proof} The corresponding cdf of the pdf in \eqref{eq.fegl2} is
\[{F\left(x\right)=\left(1-{(1-t)}^{G\left(\psi \right)}\right)}^{\alpha }.\]
So, $G\left(X\right)/G\left(\psi \right)$ is distributed as
$1-{\left(1-U^{1/\alpha }\right)}^{1/G\left(\psi \right)}$ where $U$ has a standard uniform distribution. Therefore, $n^{1/G(\psi )}(\frac{G\left(X_{\left(n\right)}\right)}{G\left(\psi \right)}-1)$ is distributed as $-n^{1/G\left(\psi \right)}{(1-U^{1/\alpha }_{\left(n\right)})}^{1/G\left(\psi \right)}$where $U_{(n)}$ has a beta distribution with parameters $n$ and 1. Now, we have
\[{\mathop{\lim }_{n\to \infty } P\left(n(1-U^{1/\alpha }_{\left(n\right)})\le x\right) }=1-{\rm e}^{-\alpha x},\]
and then $n^{1/G\left(\psi \right)}\left(G(X_{\left(n\right)})-G\left(\psi \right)\right){{\stackrel{\ \ \ d\ \ \ }{\longrightarrow}}}-G\left(\psi \right)V^{1/G\left(\psi \right)}$. By using the delta method and the relation between the derivatives of the $G\left(x\right)$ and $G^{-1}\left(x\right)$, the proof is completed.
\end{proof}

\begin{theorem}
a) Conditioning on $X_{\left(n\right)}$, the asymptotic distribution of the modified MLE, $(\tilde{\alpha },\tilde{a},\tilde{b})$ is multivariate normal distribution.\\
b)The asymptotic distribution of $(\tilde{\alpha },\tilde{a},\tilde{b})$ is (i) multivariate normal if $G\left(\psi \right)<\frac{1}{2}$, (ii) multivariate Weibull if $G\left(\psi \right)>\frac{1}{2}$, and (iii) a mixture of normal and Weibull if $G\left(\psi \right)=\frac{1}{2}$.
\end{theorem}

\begin{proof}
 For the proof of this theorem, reader can be refer to
 \cite{ku-gu-11}.
\end{proof}

\section{ A real example}
\label{sec.ex}

The following data set is given by
\cite{Aarset-87}
and represents the lifetimes of 50 devices.
\cite{sa-ku-09}
 and
\cite{si-or-co-10}
 also analyzed this data.

\begin{center}
0.1 0.2 1 1 1 1 1 2 3 6 7 11 12 18 18 18 18 1821 32 36 40 45 46 47 50 55

 60 63 63 67 67 67 67 72 7579 82 82 83 84 84 84 85 85 85 85 85 86 86
\end{center}
Using this data, we obtained the MLE's of parameters of six distributions: extended generalized linear failure rate (EGLFR), extended generalized exponential (EGE), extended generalized Rayleigh (EGR),  generalized linear failure rate (GLFR),  generalized exponential (GE), and  generalized Rayleigh (GR)   distributions. The results are given in  Table \ref{tab.ex}.

Based on the MLE's of parameters, we calculated minus of log-likelihood function \break ($-\log⁡(L)$), Kolmogorov-Smirnov (K-S) statistic with its p-value, Akaike information criterion (AIC), Akaike information criterion corrected (AICC), Bayesian information criterion (BIC) and likelihood ratio test (LRT) with its p-value. Table \ref{tab.ex} indicates that the GLFR and GE models are not suitable for this data set based on K-S statistic. Also, The EGLFR model has the lowest $-\log ⁡(L)$, AIC, AICC and BIC values among all fitted models. A comparison of the proposed distribution with some of its sub-models using p-value of LRT shows that the EGLFR model yields a better fit than the other five distributions to this real data set.

The beta modified Weibull (BMW) distribution was introduced by
\cite{si-or-co-10}
and has the following pdf:
\[f_{BMW}\left(x\right)=\frac{\alpha x^{\gamma }\left(\gamma +\beta x\right)}{B\left(a,b\right)}{\rm e}^{\beta x-b\alpha x^{\gamma }{\rm e}^{\beta x}}{\left(1-{\rm e}^{-\alpha x^{\gamma }{\rm e}^{\beta x}}\right)}^{a-1},\ \ \ \ x>0,\]
where $B\left(a,b\right)$ is the beta function with $\alpha ,\beta ,\gamma >0$ and $0<a,b<1$. This distribution contains several important distributions such as Weibull, modified Weibull \citep{la-xi-mu-03}, generalized exponential \citep{gu-ku-99}, beta Weibull \citep{le-fa-ol-07} and generalized modified Weibull \citep{ca-or-co-08} distributions.
\cite{si-or-co-10}
showed that the BMW distribution produces a better fit than its sub-models for this data set. In Table \ref{tab.ex}, we present the MLE's, $-{\log  (L)}$, and other criteria for the BMW distribution, and we can conclude that the EGLFR, EGR and EGE have better  fit than the BMW distribution for this data set.

In Figure \ref{fig.ex1}, we plot the histogram of this data set and the estimated pdf of the seven models. Moreover, the plots of empirical cdf of the data set and estimated cdf of the five models are displayed in Figure \ref{fig.ex1}. Again, we can conclude that EGLFR distribution is a very satisfactory model for this data set.

There are various extensions of GLFR distribution in literature. The pdf of seven  extensions of this distribution are given as follows. Here, we consider ${\rm g} (x ) =\left(a+bx\right){\rm e}^{- \left(ax+\frac{b}{2}x^2\right)}$ and $G(x)=1-{\rm e}^{-\left(ax+\frac{b}{2}x^2\right)}$, where $a>0$ and $b>0$.

\bigskip
\noindent\textbf{Beta LFR (BLFR) distribution}

This distribution is proposed by \cite{ja-ma-2012}
 with the following pdf:
\[f_{{\rm BLFR}}\left(x\right)=\frac{{\rm g}\left(x\right)}{B\left(\alpha ,\beta \right)}{\left(G(x)\right)}^{\alpha -1}{\left(1-G\left(x\right)\right)}^{\beta -1},\ \ \ \ x>0,\]
where $\alpha >0$ and $\beta >0$.

\bigskip
\noindent \textbf{Kumaraswamy GLFR (KGLFR) distribution }

This distribution is proposed by
\cite{elbatal-13}
 with the following pdf:
\[f_{{\rm KGLFR}}\left(x\right)=\alpha \beta {\rm g}\left(x\right){\left(G(x)\right)}^{\alpha -1}{\left(1-{\left(G(x)\right)}^{\alpha }\right)}^{\beta -1},\ \ \ \ x>0,\]
where $\alpha =c\theta $, and $\theta >0$, $c>0$ and $\beta >0$.

\bigskip
\noindent\textbf{McDonald GLFR (MCGLFR) distribution }

This distribution is proposed by
\cite{el-me-ma-14}
 with the following pdf:
\[f_{{\rm MC}{\rm GLFR}}\left(x\right)=\frac{\gamma {\rm g}\left(x\right)}{B\left(\alpha ,\beta \right)}{\left(G(x)\right)}^{\alpha \gamma -1}{\left(1-{\left(G(x)\right)}^{\gamma }\right)}^{\beta -1},\ \ \ \ x>0,\]
where $\gamma =c\theta $, $\theta >0$, $c>0$, $\alpha >0$ and $\beta >0$.

\bigskip
\noindent \textbf{Modified GLFR (MGLFR) distribution}

This distribution is proposed by
\cite{jamkhaneh-14}
 with the following pdf:
\[f_{{\rm M}{\rm GLFR}}\left(x\right)=\alpha {\left(a+b\beta x^{\beta -1}\right)e}^{-\left(ax+bx^{\beta }\right)}{\left(1-{\rm e}^{-\left(ax+bx^\beta\right)}\right)}^{\alpha -1},\ \ \ \ x>0,\]
where $a>0$, $b>0$, $\alpha >0$ and $\beta >0$.

\bigskip
\noindent\textbf{ Poisson GLFR (PGLFR) distribution}

This distribution is proposed by
\cite{co-or-le-15}
 with the following pdf:
\[f_{{\rm P}{\rm GLFR}}\left(x\right)=\frac{\beta \alpha {\rm g}\left(x\right)}{\left(1-{\rm e}^{-\beta }\right)}{\left(G(x)\right)}^{\alpha -1}{\rm e}^{-\beta {\left(G(x)\right)}^{\alpha }},\ \ \ \ x>0,\]
where $\alpha >0$ and $\beta >0$.

\bigskip
\noindent\textbf{Geometric GLFR (GGLFR)} \textbf{distribution}

This distribution is proposed by
\cite{na-sh-re-14}
with the following pdf:
\[f_{{\rm G}{\rm GLFR}}\left(x\right)=\frac{\alpha \left(1-\beta \right){\rm g}\left(x\right){\left(G(x)\right)}^{\alpha -1}}{{\left(1-\beta \left(1-{\left(G(x)\right)}^{\alpha }\right)\right)}^2},\ \ \ \ x>0,\]
where $\alpha >0$ and $\beta <1$.

\bigskip
\noindent\textbf{Generalized Linear Exponential (GLE)} \textbf{distribution}

This distribution is proposed by
\cite{Ma-al-10}
 with the following pdf:
\[f_{{\rm G}{\rm GLFR}}\left(x\right)=\alpha \left(a+bx\right){\left(ax+\frac{b}{2}x^2-\beta \right)}^{\alpha -1}{\rm e}^{-{\left(ax+\frac{b}{2}x^2-\beta \right)}^{\alpha }},\ \ \ \ x>0,\]
where $a>0$, $b>0$, $\alpha >0$ and $\beta >0$.

We also fitted these distributions to the data set, and obtained the MLE's, $-{\log  (L)\ }$ and other criteria for them. The results are given in Table \ref{tab.ex2}. It can be concluded that the EGLFR, EGR and EGE have better  fit than these seven extension of GLFR distribution for this data set. The histogram of the data set with the estimated pdf's  and the empirical cdf of the data set with estimated cdf's for distributions are given in  Figure \ref{fig.ex2}.

\begin{table}

\caption{ MLE's of the model parameters, and  the K-S, AIC, AICC, BIC, and LRT statistics.}\label{tab.ex}
\begin{center}
\begin{tabular}{|l|c|c|c|c|c|c|c|} \hline
 & \multicolumn{7}{|c|}{Distribution} \\ \hline
Statistic       & EGLFR  & EGE    & EGR & GLFR   & GE      & GR &  BMW\\ \hline
$\hat{\alpha }$ &             0.2620  & 0.5368 &0.2590& 0.5327 & 0.7798  & 0.3520& 0.0002 \\
$\hat{\beta }$  &             4.5000  & 1.8199 &4.2100& ---    & ---     & ---& 0.0541\\
$\hat{a}$      &$1.21\times {10}^{-8}$& 0.0064 & ---    & 0.0038 & 0.0187&---&0.1975 \\
$\hat{b}$       &  0.00006            & ---    &0.00006 & 0.0003 & ---   &0.0003 &0.1647\\
$\hat{\gamma}$  & ---                 & ---    &---     & ---    & ---   & ---&1.3771\\\hline
$-{\log  (L)\ }$& 173.9487           &189.1973&180.5367&233.1447&239.9951&234.5655&220.6601\\
K-S             & 0.0981             & 0.1558  &0.0872 & 0.1832 & 0.2042 &0.2011&0.0846 \\
p-value (K-S)   & 0.7215             & 0.1763  &0.8413 & 0.0696& 0.0309 &0.0350 &0.3971 \\
AIC             & 355.8974          &384.3945&367.0733&472.2895&483.9903 &473.1309&451.3201\\
AICC            & 356.7863          &384.9163&367.5951&472.8110&484.2456 &473.3862&452.6838\\
BIC             & 363.5455          & 390.1306&372.8094&478.0256&487.8143&476.9550&460.8802\\
LRT             & ---                 & 30.4971&13.1759 &118.3921&132.0929 &121.2335&---\\
p-value (LRT)   & ---                 & 0.0000  &0.0002  &0.0000  &0.0000   &0.0000&--- \\ \hline
\end{tabular}
\end{center}
\end{table}

\begin{figure}
\centering
\includegraphics[width=7.7cm,height=7.7cm]{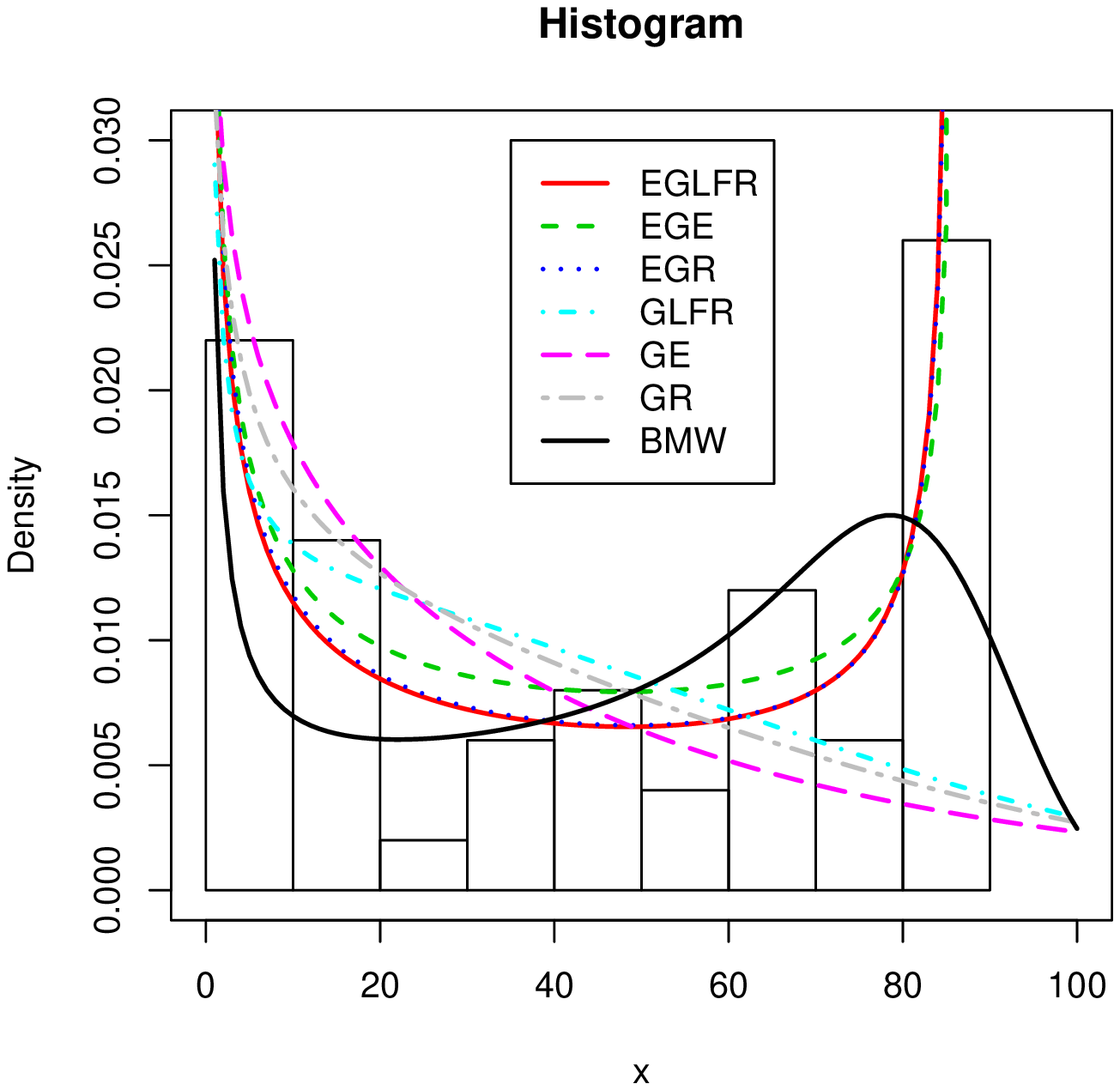}
\includegraphics[width=7.7cm,height=7.7cm]{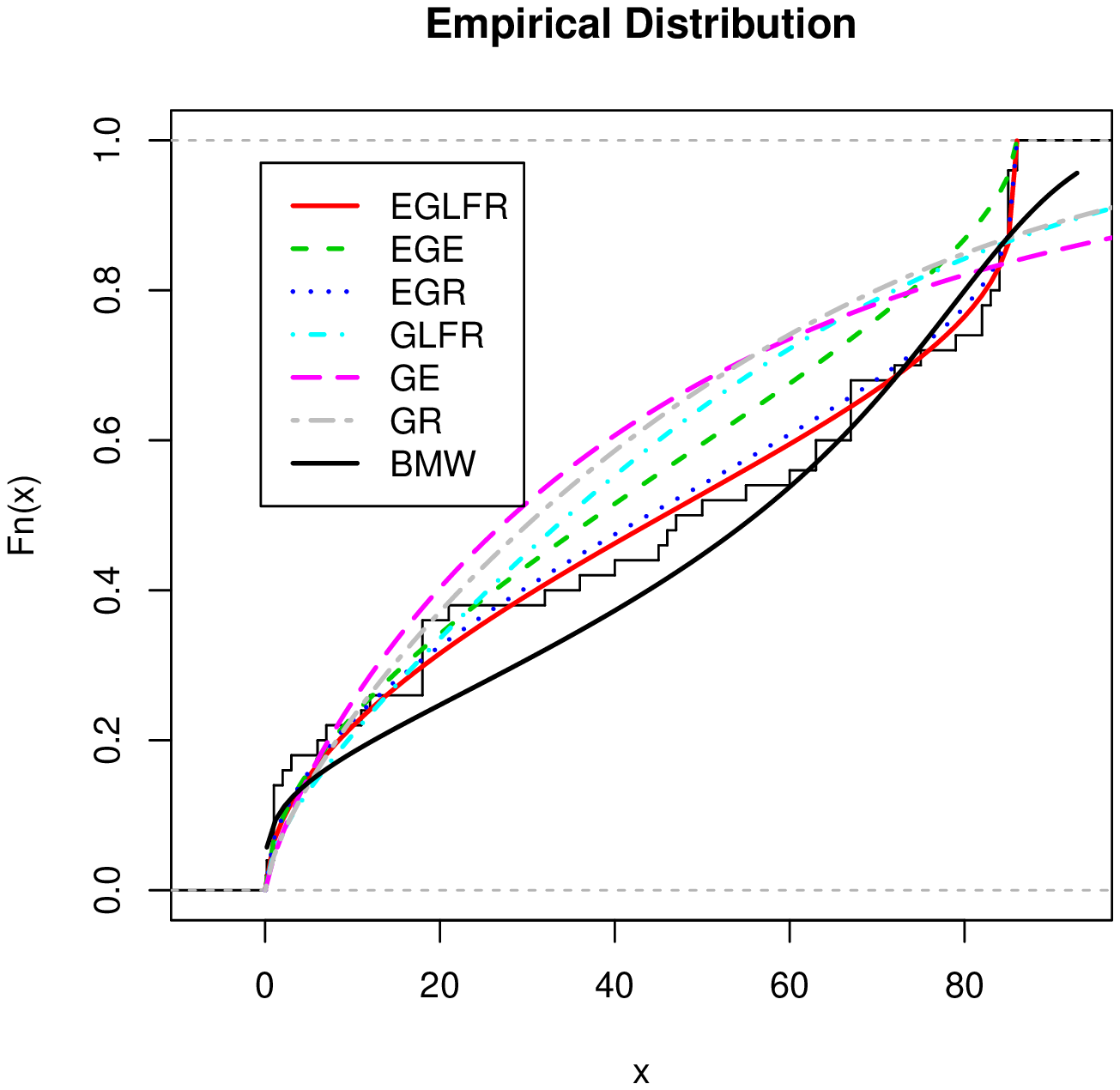}
\vspace{-0.9cm}
\caption{The histogram of the data set with the estimated pdf's (left); the empirical cdf of the data set and estimated cdf's (right). }\label{fig.ex1}

\end{figure}

\begin{table}

\caption{MLEs of the model parameters, and the K-S (and its p-value), AIC, AICC, BIC for other extensions of GLFR distribution.}\label{tab.ex2}
\begin{center}
\begin{tabular}{|l|c|c|c|c|c|c|c|} \hline
 & \multicolumn{7}{|c|}{Distribution} \\ \hline
Statistic & BLFR & KGLFR & MCGLFR & MGLFR & PGLFR & GGLFR & GLE \\ \hline
$\hat{\alpha }$ & 0.3347 & 0.6525 & 0.0295 & 19699.45 & 0.5327 & 0.2624 & 0.6262 \\
$\hat{\beta }$ & 0.1243 & 0.0622 & $5.74\times {10}^8$ & 0.0164 & ${10}^{-8}$ & -5.5536 & 0.0015 \\
$\hat{a}$ & 0.0172 & 0.2988 & 0.0015 & 0.0246 & 0.0038 & 0.0086 & 0.0149 \\
$\hat{b}$ & 0.0035 & 0.0007 & $6.66\times {10}^{-5}$ & 8.8393 & 0.0003 & 0.0005 & 0.0005 \\
$\hat{\gamma} $ & --- & --- & 1.8936 & --- & --- & --- & --- \\ \hline
$-{\log  (L)\ }$ & 230.3785 & 238.0490 & 221.9929 & 235.3460 & 233.1447 & 229.9373 & 227.1663 \\
K-S & 0.1554 & 0.1666 & 0.1949 & 0.1624 & 0.1832 & 0.1297 & 0.2327 \\
p-value (K-S) & 0.1784 & 0.1246 & 0.0448 & 0.1428 & 0.0696 & 0.3694 & 0.0088 \\
AIC & 468.7570 & 484.0980 & 453.9858 & 478.6921 & 474.2895 & 467.8745 & 462.3327 \\
AICC & 469.6459 & 484.9869 & 455.3494 & 479.5810 & 475.1784 & 468.7634 & 463.2216 \\
BIC & 476.4051 & 491.7461 & 463.5459 & 486.3402 & 481.9376 & 475.5226 & 469.9808 \\ \hline
\end{tabular}
\end{center}
\end{table}

\begin{figure}
\centering
\includegraphics[width=7.7cm,height=7.7cm]{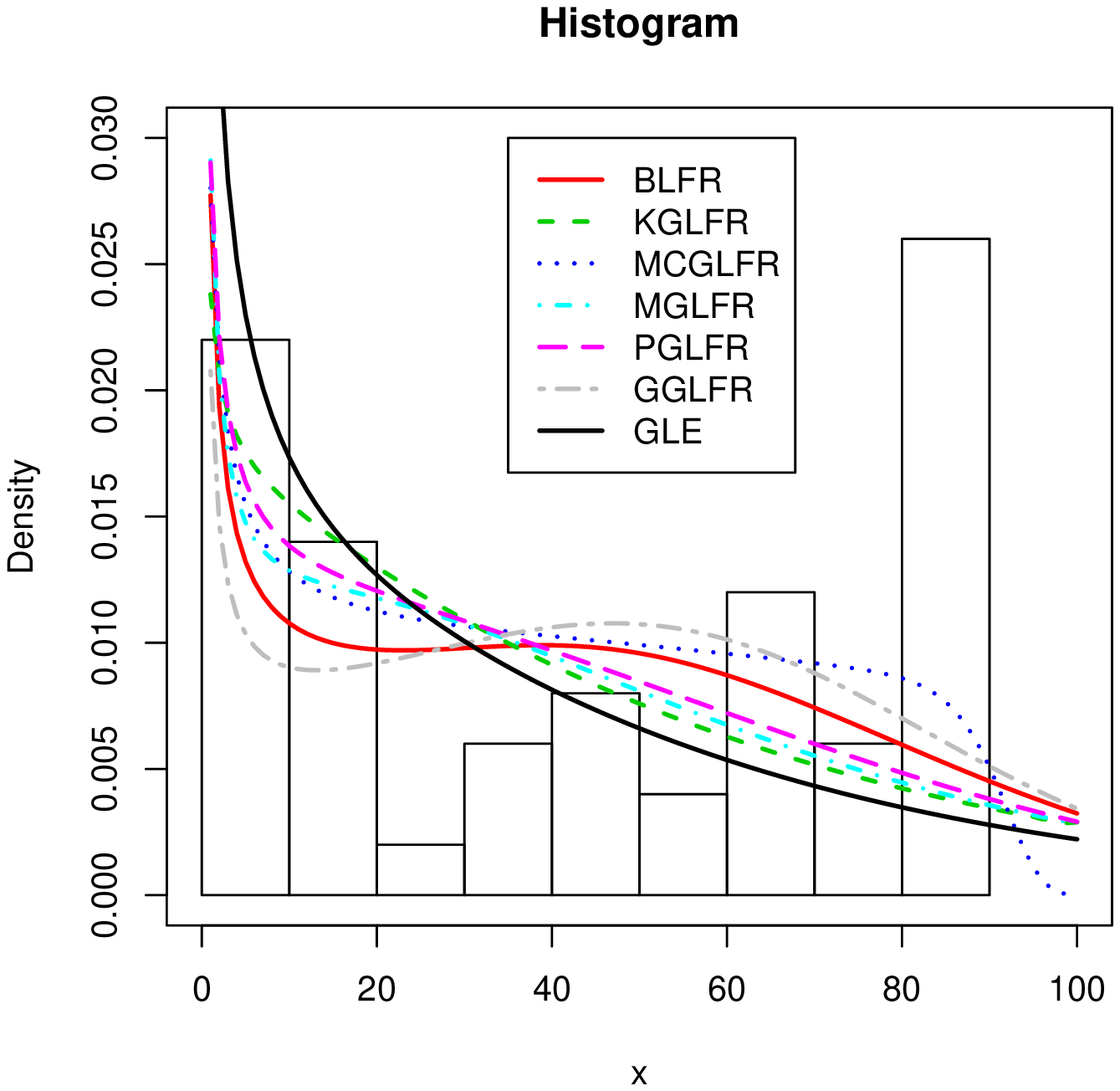}
\includegraphics[width=7.7cm,height=7.7cm]{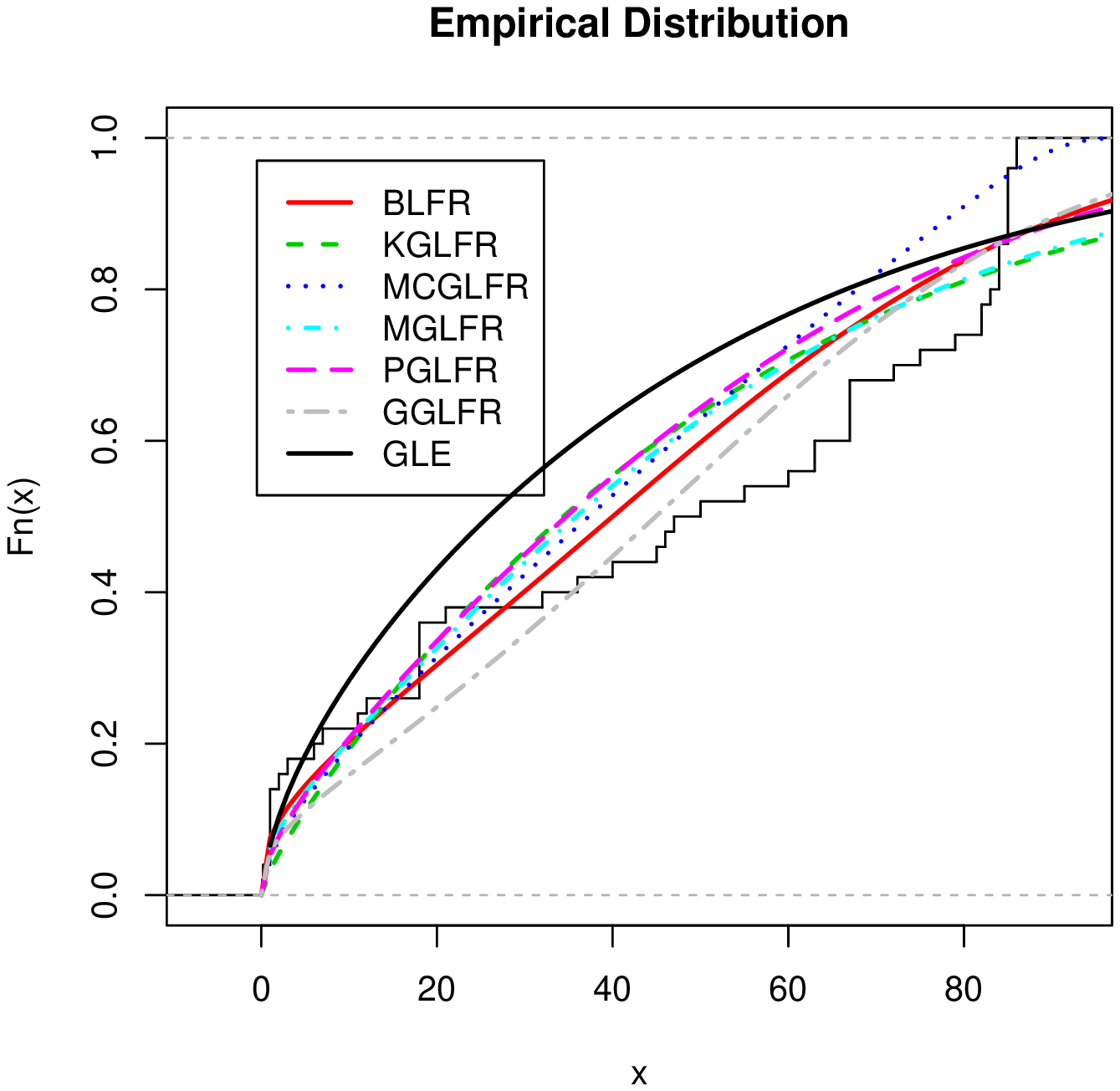}
\vspace{-0.9cm}

\caption{The histogram of the data set with the estimated pdf's (left); the empirical cdf of the data set and estimated cdf's (right) for seven extensions of GLFR distribution. }\label{fig.ex2}

\end{figure}

\section{Conclusion}
\label{sec.con}
In this paper, a four-parameter EGLFR distribution has been proposed. Although the GLFR distribution cannot have a unimodal hazard rate function, but EGLFR distribution can have a decreasing, increasing, unimodal, and bathtub-shaped hazard rate function depending on its parameters. This new class of distributions includes some well-known distribution such as EGE and GLFR distributions as special sub-models. Several properties of EGLFR distribution such as quantiles function, Skewness and kurtosis measures, and the $r$-th non-central moment have been given. the maximum likelihood estimators of the parameters are given and also, the asymptotic distribution of estimates are investigated. By using a real data set, we have shown that our new distribution fits to the lifetime data very well than the well-known distributions.
Recently
\cite{sa-ha-sm-ku-11}
introduced the bivariate GLFR distribution and its multivariate extension. So, one can use the approach used in this paper to develop the multivariate EGLFR distribution.

\section*{Acknowledgements}
The authors would like to thank the anonymous  referees  for many helpful comments and  suggestions.

\bibliographystyle{apa}

\end{document}